\documentclass{article}
\usepackage[utf8]{inputenc}
\usepackage{amsmath,amssymb,dsfont,stfloats,color,url,bbm}
\usepackage{subfigure}
\usepackage{dsfont}
\usepackage{cite}
\usepackage{xcolor}
\usepackage{pgfplots, pgfplotstable}

\allowdisplaybreaks

\usepackage{mathtools}
\usepackage{cite}
\usepackage[margin=1in]{geometry}
\usepackage{amsmath}
\usepackage{amssymb}
\usepackage{amsthm}
\usepackage{latexsym}
\usepackage{verbatim}
\usepackage{subfigure}
\usepackage{psfrag}
\usepackage{color}
\usepackage{enumitem}

\usepackage[utf8]{inputenc} 
\usepackage{amsfonts} 
\usepackage[T1]{fontenc}
\usepackage{url}              

\usepackage{hyperref}
\hypersetup{
	colorlinks=true,
	linkcolor=blue,
	filecolor=blue,
	citecolor = blue,      
	urlcolor=cyan,}

\usepackage{esdiff}
\usepackage{pgf,tikz,psfrag}
\usepackage{yfonts}

\newtheorem{theorem}{Theorem}
\newtheorem{proposition}{Proposition}
\newtheorem{lemma}{Lemma}

\newtheorem{assumption}{Assumption}

\newtheorem{definition}{Definition}
\newtheorem{remark}{Remark}

\def\mathlette#1#2{{\mathchoice{\mbox{#1$\displaystyle #2$}}%
{\mbox{#1$\textstyle #2$}}%
{\mbox{#1$\scriptstyle #2$}}%
{\mbox{#1$\scriptscriptstyle #2$}}}}
\newcommand{\matr}[1]{\mathlette{\boldmath}{#1}}

\providecommand{\GS}[2]{\mathcal{GS}(#1 \mid #2)}
\newcommand{\Op}[1]{\mathcal{O}(#1)}
\newcommand{\Opn}[1]{\mathcal{O}_\theta(#1)}

\newcommand{\psiv}{\matr\psi}

\providecommand{\norm}[1]{\lVert#1\rVert}
\newfont{\bbb}{msbm10 scaled 700}

\newfont{\bb}{msbm10 scaled 1100}

\newcommand{\PP}{\mbox{\bb P}}
\newcommand{\RR}{\mbox{\bb R}}

\newcommand{\NN}{\mbox{\bb N}}

\newcommand{\av}{{\bf a}}
\newcommand{\bv}{{\bf b}}

\newcommand{\ev}{{\bf e}}

\newcommand{\gv}{{\bf g}}

\newcommand{\qv}{{\bf q}}

\newcommand{\uv}{{\bf u}}

\newcommand{\vv}{{\bf v}}
\newcommand{\xv}{{\bf x}}

\newcommand{\zv}{{\bf z}}

\newcommand{\Am}{{\bf A}}
\newcommand{\Bm}{{\bf B}}
\newcommand{\Cm}{{\bf C}}

\newcommand{\Em}{{\bf E}}

\newcommand{\Gm}{{\bf G}}

\newcommand{\Id}{{\bf I}}

\newcommand{\Km}{{\bf K}}

\newcommand{\Om}{{\bf O}}
\newcommand{\Pm}{{\bf P}}

\newcommand{\Um}{{\bf U}}

\newcommand{\Xm}{{\bf X}}
\newcommand{\Ym}{{\bf Y}}

\newcommand{\gammav}{\hbox{\boldmath$\gamma$}}
\newcommand{\deltav}{\hbox{\boldmath$\delta$}}

\newcommand{\Deltam}{\hbox{\boldmath$\Delta$}}

\newcommand{\Psim}{\hbox{\boldmath$\Psi$}}

\newcommand{\diag}{{\hbox{diag}}}

\renewcommand{\Im}{{\rm Im}}

\usepackage{stmaryrd} 

\usepackage{graphicx} 
\usepackage{stmaryrd} 

\title{Inverse of a Sum of Random Matrices: \\A Dynamical Mean-Field Approach}
\author{{Burak~\c{C}akmak} and Manfred Opper \thanks{The authors contributed equally.}
\thanks{The work of B. \c{C}akmak was supported by the Gottfried Wilhelm Leibniz-Preis 2021 of DFG.}
\thanks{The authors are with Technical University of Berlin, Germany (e-mail: \{burak.cakmak, manfred.opper\}@tu-berlin.de).}}

\date{\today}

\begin{document}
\maketitle

\begin{abstract}
We study the inverse $(\Om^\top\Am\Om+\Bm)^{-1}$, where
$\Am,\Bm\in\mathbb R^{N\times N}$ are possibly random symmetric
positive definite matrices and $\Om$ is a Haar-distributed
orthogonal matrix. We assume that, as $N\to\infty$, the
empirical eigenvalue distributions of $\Am$ and $\Bm$ converge
almost surely to deterministic limits, denoted by
${\rm F}_{\Am}$ and ${\rm F}_{\Bm}$. We also assume that their
smallest eigenvalues remain bounded away from zero and their
largest eigenvalues remain bounded, almost surely, as
$N\to\infty$. We analyze the deterministic-equivalent-type
approximation
$(\Om^\top\Am\Om+\Bm)^{-1}\approx(\Bm-z^\star\Id_N)^{-1}$.

First, suppose that $\Om$ is independent of $(\Am,\Bm,\uv)$,
where $\uv\in\mathbb R^N$ satisfies $\|\uv\|=1$. There exists a sequence of standard Gaussian
vectors $\gv_N\in\RR^N$ independent of $(\Am,\Bm,\uv)$~such~that
\[
\lim_{N\to\infty}\left\Vert \,
(\Bm-z^\star\Id_N)
\left[
(\Om^\top \Am \Om + \Bm)^{-1}
-
(\Bm-z^\star\Id_N)^{-1}
\right]
(\Bm-z^\star\Id_N)\uv-
\sqrt{\frac \tau N}\,\gv_N\right\Vert\overset{\rm a.s.}{=}0\;.
\]
Here $z^\star\doteq-{\rm R}_{\Am}(-\chi)$ and $\tau\doteq{\rm R}_{\Am}'(-\chi)\bigl(1+\eta\,{\rm R}_{\Am}'(-\chi)\bigr)$ with  ${\rm R}_\Am$ denoting the
R-transform of ${\rm F}_{\Am}$ and ${\rm R}_{\Am}'$ its derivative. Writing ${\rm F}$ for the
free additive convolution of ${\rm F}_{\Am}$ and ${\rm F}_{\Bm}$,
we define
\[
(\chi,\eta)\doteq\int(x^{-1},x^{-2})\,{\rm dF}(x)\;.
\]

Second, consider a test matrix
\begin{equation}
\Km=\sum_{i=1}^{\lfloor N^q\rfloor}\lambda_i\uv_i\vv_i^\top,
\qquad q\in[0,C],\nonumber 
\end{equation}
where $C$ is fixed, $\|\uv_i\|=\|\vv_i\|=1$ and
$\limsup_{N\to\infty}\max_i|\lambda_i|\overset{\rm a.s.}{<}\infty$.
\emph{No orthogonality} is assumed among vectors in $\mathcal M\doteq
\{\uv_i,\vv_i\}_{i\leq \lfloor N^q\rfloor}$.
If $\Om$ is independent of $(\Am,\Bm,\mathcal M)$, then
\[
\lim_{N\to \infty}\frac1{N^q}
\operatorname{tr}\!\left(
\Km\Bigl[
(\Om^\top \Am \Om + \Bm)^{-1}
-
(\Bm-z^\star\Id_N)^{-1}
\Bigr]
\right)
\overset{\mathrm{a.s.}}{=}0\;.
\]
The cases $q=0$ and $q=1$ are the conventional anisotropic local and normalized-trace test cases;
the result also covers the intermediate range $q\in(0,1)$ and the extensive regime $q>1$.

We prove both results using a dynamical mean-field approach.
The generating functions of the free Appell polynomials
of $\Am$ and $\Bm$ define a linear discrete-time dynamics that
approximates the quantity of interest in the iterated large-$N$,
large-time limit. We analyze its finite-time fluctuations and
use the resulting covariance contraction to obtain the two laws.
\end{abstract}
\section{Introduction}
We address a deterministic-equivalent-type approximation~\cite{Hachem07} of the inverse of a sum of random matrices
\begin{equation}
(\Om^\top\Am\Om+\Bm)^{-1}\approx (\Bm-z^\star\Id_N)^{-1} \nonumber
\end{equation}
where $\Am\equiv \Am_N,\Bm\equiv \Bm_N\in\mathbb{R}^{N\times N}$ are
symmetric positive definite, possibly random, and $\Om\equiv \Om_N\in \RR^{N\times N}$ is Haar-distributed, independent of $\Am$ and $\Bm$. Note that this problem implicitly addresses the resolvent of the additive model $(\Om^\top\Am\Om+\Bm-z\Id_N)^{-1}$ evaluated at a spectral parameter $z\leq 0$.

Understanding such a deterministic equivalent approximation is not only a fundamental problem in random matrix theory \cite{kargin2015subordination, Alex14,knowles2017anisotropic,bao2017local}, but is also increasingly relevant to data science \cite{bun2017cleaning,couillet2022random}. As a motivating example, approximate Bayesian inference frequently relies on iterative algorithms that require $(\Om^\top\Am\Om+\Bm)^{-1}$, where $\Bm$ is typically diagonal and updated at every iteration \cite{williams2006gaussian}. Repeatedly inverting this matrix becomes a computational bottleneck for large $N$; replacing it with a deterministic-equivalent approximation allows this inversion to be bypassed altogether \cite{cakmak2018expectation}.

A simpler version of the problem was addressed in \cite[Theorem~2.17]{couillet2022random} and \cite{couillet2012random} for the case in which $\Am$ is a projection matrix, i.e., all nonzero eigenvalues of $\Am$ equal $1$.
In particular, in \cite{couillet2022random}, the Stein-method approach for Haar matrices \cite{pastur2011eigenvalue} is used to obtain deterministic equivalents for the two test cases as $N\to \infty$
\begin{align}
\frac{1}{N}{\rm tr}\left(\Km\left[(\Om^\top\Am\Om+\Bm)^{-1}- (\Bm-z^\star \Id_N)^{-1}\right]\right)&\overset{\rm a.s.}{\to}0 
\label{p1}\\
\vv^\top\left[(\Om^\top\Am\Om+\Bm)^{-1}-(\Bm-z^\star\Id_N)^{-1}\right]\uv&\overset{\rm a.s.}{\to }0 \;\label{p2}
\end{align}
for a test matrix $\Km $ with a bounded spectral norm and unit-norm vectors $\uv,\vv\in \RR^N$  which are all independent of $\Om$. 
The proof of \eqref{p1} (with $\Am$ a projection matrix) is a straightforward application of the Stein method \cite{pastur2011eigenvalue}. The proof of the latter is a considerably more involved application of the Stein method, as it requires additional control of higher-order moments; see \cite{couillet2022random} for details.

Beyond these test cases, it is natural to ask what happens when the test matrix
$\Km$ is a projection onto an $O(N^{c})$-dimensional subspace, $c<1$, so that
the quantity of interest becomes
\[
  \frac{1}{N^{c}}\operatorname{tr}\big(\Km(\Om^\top\Am\Om+\Bm)^{-1}\big),
\]
 or when $\Km$ takes the extensive form $\Km=\sum_{i=1}^{\lfloor N^{p'}\rfloor}\Km_i$ 
for some $p'>1$ where each $\Km_i$ is a projection matrix onto an $O(N^{c_i})$-dimensional subspace with $c_i\leq 1$. Both are captured by a unified formulation: we take
\[
  \Km=\sum_{i=1}^{\lfloor N^{q}\rfloor}\lambda_i\uv_i\vv_i^\top,
  \qquad q\in[0,C],
\]
with $C$ fixed and $\|\uv_i\|=\|\vv_i\|=1$, and we ask for the limit of
$N^{-q}\operatorname{tr}(\Km[\,\cdot\,])$. We stress that \emph{no orthogonality} is assumed among vectors in $\mathcal M\doteq
\{\uv_i,\vv_i\}_{i\leq \lfloor N^q\rfloor}$. Taking
$q=0$ recovers the anisotropic test case~\eqref{p2}, and $q=1$ the test
case~\eqref{p1}, while $q>1$ reaches the extensive regime.

While the Stein method might be useful for extending \eqref{p1} (or \eqref{p2}) to a general $\Am$, the aforementioned test formulation requires control of high-order moments. This indicates the need for \emph{non-asymptotic concentration inequalities} that directly target \eqref{p2}.

A deterministic equivalent estimate of \eqref{p2} is known as an \emph{anisotropic local law} \cite{knowles2017anisotropic}. Based on delicate resolvent-identity analyses, \cite{Alex14,knowles2017anisotropic}  provide non-asymptotic concentration inequalities for the \emph{anisotropic local law} when $\Am$ follows a Wigner or Marchenko--Pastur ensemble. For the additive Haar model considered here, however, 
existing local laws \cite{kargin2015subordination,bao2017local} control only the individual entries of the inverse corresponding to the special case where $\uv$ and $\vv$ are standard basis vectors (and also $\Am$ and $\Bm$ are treated deterministically).


\subsection{Proof Strategy: A Dynamical Mean-Field Argument}\label{sec:method}
In this paper, we take a novel, dynamical approach to this problem which does not rely
on the Stein method or resolvent identities.  In a first step, we will analyze the statistics of the random vector
\begin{equation}
\gammav\doteq\sqrt{N}\mathds{1}_{\mathcal E_N}\,
(\Bm-z_N^\star\Id_N)
\left[
(\Om^\top\Am\Om+\Bm-\epsilon^\star_N\Id_N)^{-1}-(\Bm-z_N^\star\Id_N)^{-1}
\right]
(\Bm-z_N^\star\Id_N)\uv  \label{gammav0}
\end{equation}
for a generic unit-norm vector $\uv$ independent of $\Om$.  Here, $\mathds{1}_{\mathcal E_N}\in\{0,1\}$ denotes the indicator of the event ${\mathcal E_N}\doteq \{\vert\epsilon_N^\star\vert\leq \sigma_{\min}(\Bm)\}$ with $\sigma_{\min}(\Bm)$ denoting the minimum singular value of $\Bm$. The explicit constructions of the scalars $z_N^\star$ and $\epsilon_N^\star$, which satisfy $z_N^\star\overset{\rm a.s.}{\to} z^\star$ and $\epsilon_N^\star \overset{\rm a.s.}{\to} 0$ as $N\to \infty$, are given later in the proof. The purpose of introducing the perturbed model with $\epsilon_N^\star$ is to obtain certain non-asymptotic concentration inequalities, which will play a key role in the proof. From the asymptotic statistics of $\gammav$, the limit of
$N^{-q}\operatorname{tr}(\Km[\,\cdot\,])$ will then be deduced in a second step. 

The main idea of our approach is to construct a temporal evolution of random vectors $\gammav^{(t)}$, $t=0,1,2,\dots$ that is devised to converge (as $t\to \infty$) to the desired vector $\gammav$
of \eqref{gammav0} and which can be analyzed in a simpler way. This approach has been motivated by methods from the statistical physics of disordered systems, where the properties of static or equilibrium states of a high-dimensional system with "frozen'' randomness are of interest.
The so-called dynamical mean-field method approaches this problem by first working out the asymptotic statistics of 
a corresponding dynamical model and then performing the large time limit to recover the static properties.
The success of this method relies on the often simple asymptotic decoupling of individual degrees of freedom (the components of random vectors) in the large $N$ limit. Applications of this approach range from models of spin glasses \cite{sompolinsky1981time,Eissfeller} and neural networks \cite{Mignacco_2021} to message-passing algorithms
\cite{opper2016theory,ccakmak2020dynamical} for high-dimensional models of statistical inference. While the original methods used in the physics literature may not be made rigorous in an easy way, recent mathematical work has justified the corresponding results in many cases, see e.g. \cite{Arous1998,Bolthausen,bayati2011dynamics,rangan2019vector,takeuchi2019rigorous,fan2022approximate,dudeja2023universality,bao2025leave}. 
In our work, we will rely on the "Householder Dice'' approach of \cite{lu2021householder} which has previously been applied, via a dynamical approach, to the statistical analysis of classification and communication-theoretic problems \cite{Cakmakisit24,cakmakit25,cakmak2026orthogonalapproximatemessagepassing}.

For our specific random matrix problem, the discrete-time dynamics designed to converge to \eqref{gammav0} is particularly simple, and in fact linear: starting from an appropriate random initialization $\gammav^{(0)}$, the iterates are given, for $t\geq 1$, by
\begin{align}
\gammav^{(t)}=\Om^\top \Em_{A}\Om\left[\Em_B\gammav^{(t-1)}+\frac{\sqrt N\, \uv}{\chi_N}\right]\;.
\end{align}
Here, $\Em_{A}\equiv \Em_A(\chi_N)$ and $\Em_{B}\equiv \Em_B(\chi_N)$ are the centered generating functions of the so-called \emph{free Appell polynomials} \cite{anshelevich2004appell} of $\Am$ and $\Bm$
evaluated at $\chi_N$. Namely, for a positive semi-definite matrix $\Xm\equiv \Xm_N\in \RR^{N\times N}$, we define the
centered generating function of its free Appell polynomials \cite{anshelevich2004appell} as
\begin{align}
\Em_{X}(s)\doteq \frac{1}{s}\left(\Xm-z_{\Xm}^N{(s)}\Id_N\right)^{-1}-\Id_N,\qquad 0<s<\frac{1}{N}{\rm tr}(\Xm^{-1})\;,
\end{align}
where $z_{\Xm}^N(s)$ is the functional inverse of the empirical Stieltjes transform $s_{\Xm}^N(z) \doteq \frac{1}{N}{\rm tr}\bigl((\Xm - z \Id_N)^{-1}\bigr)$ for $z<0$ and by convention we set $\frac{1}{N}{\rm tr}(\Xm^{-1})=\infty$ when $\Xm$ is singular.
Moreover, we introduce 
\[
\chi_N\doteq\frac{1}{N}{\rm tr}\bigl((\Om_{\rm new}^\top\Am\Om_{\rm new}+\Bm)^{-1}\bigr)\;,
\]
where $\Om_{\rm new}$ is a new Haar matrix independent of all other random elements, i.e., of $(\Om,\Am,\Bm,\mathcal M,\uv)$.\footnote{We introduce a new Haar $\Om_{\rm new}$, so that $\chi_N$ is equal in distribution to $\frac{1}{N}{\rm tr}\bigl((\Om^\top\Am\Om+\Bm)^{-1}\bigr) $ but independent of $\Om$.}
 Note  that $\chi_N<\min(\frac{1}{N}{\rm tr}(\Am^{-1}),\frac{1}{N}{\rm tr}(\Bm^{-1}))$, so that both $\Em_A$ and $\Em_B$ are well defined.
It is also interesting to note that these transformations are used in Tao's proof of the free additive convolution \cite[Section 2.3.4]{tao2023topics}.

Our method allows us to go one level beyond the deterministic equivalent and to obtain a precise Gaussian characterization of $\gammav$, which is new for this problem. Specifically, we control the distance between $\gammav$ and $\gammav^{(t)}$ through the identity (Lemma~\ref{P1}) 
\begin{align}
\gammav=\mathds{1}_{\mathcal E_N}\gammav^{(t)}+\mathds{1}_{\mathcal E_N}\left(\Id_N-\Om^\top\Em_{A}\Om\Em_{B}\right)^{{-1}}\left(\gammav^{(t+1)}-\gammav^{(t)}\right)\;, \label{non-asy-decom}
\end{align}
and we will verify the non-asymptotic decomposition for any $T>0$ independent of $N$ (Proposition~\ref{P2})
\begin{align}
\bigl[\gammav^{(0)},\gammav^{(1)},\cdots, \gammav^{(T)}\bigr]= \Gm\sqrt{ \mathcal{C}_{N}^{(0:T)}}+\widetilde\theta_N\matr\Psi\label{explicit}
\;.
\end{align}
Here, $\widetilde\theta_N$ is an auxiliary random variable which is almost surely bounded as $N\to\infty$. Furthermore, for each $p\in\NN$ there is a
deterministic finite constant $C_p$ (independent of $N$) such that $\mathbb E[(\norm{\matr \Psi}_{\rm F})^p]^{\frac 1 p}\leq C_p$.
The matrix $\Gm\in \RR^{N\times (T+1)}$ has i.i.d.\ standard Gaussian entries and is independent of $(\Am,\Bm,\uv,\chi_N)$. The matrix $\mathcal{C}_{N}^{(0:T')}\in\RR^{(T'+1)\times (T'+1)}$ is constructed as follows. Writing $\mathcal{C}_{N}^{(t,s)}$ for its $(t+1,s+1)$ entry and starting from the initial condition $\mathcal{C}_N^{(t,0)} = \tau \delta_{t0}$, we set recursively
\begin{align}
\label{C_N}
\mathcal{C}_{N}^{(t+1,s+1)} &= \frac{1}{N}{\rm tr}(\Em_A^2)\left[{\chi_N^{-2}}+\frac{1}{N}{\rm tr}(\Em_B^2)\,\mathcal{C}_N^{(t,s)}\right]\;.
\end{align}
We next state our main asymptotic results, which follow essentially from the decomposition \eqref{explicit} together with the contraction properties of the dynamical covariance matrix $\mathcal{C}_{N}^{(0:T')}$.

\subsection{Main Results}
 
\begin{assumption}\label{as1}
Let $\Am,\Bm$ be symmetric positive definite (random, in general) matrices (which may depend on each other, e.g., $\Am=\Bm$). Let the empirical eigenvalue distributions of $\Am$ and $\Bm$ converge almost surely to limiting deterministic distributions as $N\to\infty$. Let $\liminf_{N\to\infty}\sigma_{\min}(\Am)\overset{\rm a.s.}{>}0$ and $\limsup_{N\to\infty}\sigma_{\max}(\Am)\overset{\rm a.s.}{<}\infty$, and likewise for $\Bm$. Here, $\sigma_{\min}(\cdot)$ (or $\sigma_{\max}(\cdot)$) denotes the minimum (or maximum) singular value of the matrix in the argument.  
\end{assumption}

Note that with $\Om$ being Haar independent of $(\Am,\Bm)$ Assumption~\ref{as1} implies that $\Om^\top\Am\Om$ and $\Bm$ are almost surely
asymptotically free\cite{hiai2000asymptotic}, so the model considered here is a sum of asymptotically free matrices and the limiting eigenvalue distribution of the sum is given by the free
additive convolution ${\rm F}\doteq{\rm F}_{\Am}\boxplus{\rm F}_{\Bm}$. 
 
Then, under Assumption~\ref{as1}, we introduce the asymptotic quantities
\begin{align}\label{chieta}
    (\chi, \eta) &\doteq \int (x^{-1},x^{-2})\, {\rm d}{\rm F}(x)\;.
\end{align}
Note that $\chi,\eta<\infty$, since $\liminf_{N\to\infty}\sigma_{\min}(\Om^\top\Am\Om+\Bm)\overset{\rm a.s.}{>}0$. We further define
\begin{align}
    z^\star&\doteq-{\rm R}_{\Am}(-\chi)<0\\
    \tau &\doteq {\rm R}_{\Am}'(-\chi)
    \bigl(1 + \eta\, {\rm R}'_{\Am}(-\chi)\bigr)\geq 0\;,\label{tau}
\end{align}
where ${\rm R}_{\Am}$ is the R-transform (see Definition~\ref{defR})  of the limiting distribution ${\rm F}_{\Am}$ and ${\rm R}'_{\Am}$ is its derivative.
 
\begin{theorem}[Deterministic Equivalent Law]\label{th1}
Let Assumption~\ref{as1} hold. Consider test matrix
\[
\Km=\sum_{i=1}^{\lfloor N^q\rfloor}\lambda_i\uv_i\vv_i^\top,
\qquad q\in[0,C],
\]
where $C$ is a constant, $\|\uv_i\|=\|\vv_i\|=1$  and $\limsup_{N\to\infty}\max_i|\lambda_i|\overset{\rm a.s.}{<}\infty$. \emph{No orthogonality} is assumed among vectors in $\mathcal M\doteq
\{\uv_i,\vv_i\}_{i\leq \lfloor N^q\rfloor}$.
If $\Om$ is independent of $(\Am,\Bm,\mathcal M)$,  then
\[
\lim_{N\to \infty}\frac1{N^q}
\operatorname{tr}\!\left(
\Km\Bigl[
\left(\Om^\top \Am \Om + \Bm\right)^{-1}
-
(\Bm-z^\star\Id_N)^{-1}
\Bigr]
\right)
\overset{\mathrm{a.s.}}{=}0\;.
\]
\end{theorem}
 
Second, we state the Gaussian characterization of the problem.
 
\begin{theorem}[Gaussian Law]\label{th2}
Let Assumption~\ref{as1} hold and let $\Om$ be a Haar orthogonal matrix  independent of $(\Am,\Bm,\uv)$ for some unit-norm vector $\uv$. Then, there exists a sequence of standard Gaussian random vectors $\gv_N\in \RR^N$ independent of $(\Am,\Bm,\uv)$ such that  
\[
\lim_{N\to\infty}\left\Vert\,
(\Bm-z^\star\Id_N)
\left[
(\Om^\top \Am \Om + \Bm)^{-1}
-
(\Bm-z^\star\Id_N)^{-1}
\right]
(\Bm-z^\star\Id_N)\uv-
\sqrt{\frac\tau N}\,\gv_N\right\Vert\overset{\rm a.s.}{=}0\;.
\]
\end{theorem}

\subsection{Organization}
The remainder of the paper is organized as follows.
In Section~\ref{sec_proof}, we prove both Theorem~\ref{th1} and Theorem~\ref{th2}. 
The proof of the underlying non-asymptotic analysis of the proposed dynamics $\gammav^{(t)}$ (i.e., Proposition~\ref{P2}) is given in Section~\ref{ProofP2}.
In Section~\ref{outlook} we provide an outlook. 
Unless they are lengthy, the proofs of auxiliary technical lemmas are given immediately; otherwise, they are deferred to the appendices.
\subsection{Notations}
We denote the normalized trace of a matrix
$\Xm\equiv\Xm_{N}\in \mathbb{R}^{N \times N}$ by
\begin{align}
    \phi_{N}(\Xm) \doteq \frac{1}{N}{\rm tr}(\Xm)\;.\nonumber
\end{align}
For $\mathbf a, \mathbf b \in \mathbb{R}^N$, we denote their normalized inner product by
\[
\langle \mathbf a, \mathbf b \rangle \doteq  \frac{1}{N}\mathbf a^\top \mathbf b .
\]
We write $\Id_N$ for the $N\times N$ identity matrix and $\delta_{ts}$ for the Kronecker delta.

\section{The Proofs of the Main Results}\label{sec_proof}
    
We begin by defining Voiculescu's $R$-transform for measures on $\RR_{+}$.
\begin{definition}\label{defR}
Let $X \geq 0$ be a random variable with distribution ${\rm F}_{X}$,
and let $\mu_{X^{-1}} = \mathbb{E}[X^{-1}]$, with the convention
$\mu_{X^{-1}} = \infty$ if ${\rm F}_X(0) > 0$. The R-transform of
${\rm F}_X$ is given as \cite{guionnet2005fourier}
\[
    {\rm R}_X(s) = z_X(-s) - s^{-1},
    \qquad s \in (-\mu_{X^{-1}},\, 0),
\]
where $z_X$ is the functional inverse of the Stieltjes transform
$s_X(z) = \mathbb{E}[(X-z)^{-1}]$ for $z \in (-\infty, 0)$. The R-transform
is analytic and, unless $X$ is constant, strictly increasing on $(-\mu_{X^{-1}}, 0)$, with boundary value
${\rm R}(-\mu_{X^{-1}}) \equiv \lim_{\omega \to -\mu_{X^{-1}}^+}
{\rm R}(\omega) = 1/\mu_{X^{-1}}$.  
\end{definition}

From Assumption~\ref{as1}, $\liminf_{N\to\infty}\sigma_{\min}(\Am)>0$ and $\Am$ has a limiting eigenvalue distribution ${\rm F}_{\Am}$ almost surely as $N\to \infty$; hence
\begin{align}
\lim_{N\to\infty}\phi_N(\Am^{-1})\overset{\rm a.s.}{=}\underbrace{\int x^{-1} {\rm dF}_{\Am}(x)}_{\doteq\mu_{\Am^{-1}}}<\infty.
\end{align}
Similarly, we have the convergence $\phi_N(\Bm^{-1})\overset{\rm a.s.}{\to} \mu_{\Bm^{-1}}$. Since $\sigma_{\min}(\Om_{\rm new}^\top\Am\Om_{\rm new}+\Bm)\geq \sigma_{\min}(\Am)+\sigma_{\min}(\Bm)$ and, by asymptotic freeness, the sum has the limiting eigenvalue distribution ${\rm F}_{\Am}\boxplus{\rm F}_{\Bm}$, the same argument gives $\chi_N\overset{\rm a.s.}{\to}\chi$. Furthermore, we note the identity, valid for $\Xm,\Ym>\matr 0$
\begin{equation}
\phi_N(\Xm^{-1}-(\Xm+\Ym)^{-1})=\phi_N((\Xm\Ym^{-1}\Xm+\Xm)^{-1})\geq \frac{1}{\sigma_{\max}(\Xm\Ym^{-1}\Xm+\Xm)}\;.
\end{equation}
Thus, from Assumption~\ref{as1}, we have the strict bound
\begin{align}
\chi<\min(\mu_{\Am^{-1}},\mu_{\Bm^{-1}})\;, \label{boundchi}
\end{align}
so that $-\chi$ is an interior point of the domains of both ${\rm R}_{\Am}$ and ${\rm R}_{\Bm}$, while it
is the left endpoint of that of ${\rm R}_{\Om^\top\mathbf{A}\Om+\mathbf{B}}$.
Then, we note the following identity
\begin{align}
    \frac{1}{\chi} \overset{(a)}{=} {\rm R}_{\Om^\top\mathbf{A}\Om+\mathbf{B}}(-\chi)
    \overset{(b)}{=} {\rm R}_{\mathbf{A}}(-\chi) + {\rm R}_{\mathbf{B}}(-\chi),
    \label{subordination}
\end{align}
where $(a)$ and $(b)$ use the definition of $\chi$ in \eqref{chieta} and
the additivity of the R-transform for asymptotically free
matrices \cite{mingo2017free}, respectively. 

Let $z_{\Am}$ denote the functional
inverse of the limiting Stieltjes transform $s_{\Am}(z)\doteq \int (x-z)^{-1}{\rm d F}_{\Am}(x)$ for $z<0$ and similarly we define $z_{\Bm}$.  Note that ${\rm R}_{\Am,\Bm}(-\chi)>0$. Then, from \eqref{subordination} we have
\begin{align}
z_{\Bm}(\chi)\equiv-{\rm R}_{\Am}(-\chi)<0 \quad \text{and}\quad z_{\Am}(\chi)\equiv -{\rm R}_{\Bm}(-\chi)<0\;. 
\end{align}

\subsection{The Large-$N$ Substitution}
For brevity, we define the deviation matrix
\[
    \matr\Delta \doteq (\Om^\top\Am\Om+\Bm)^{-1} - \bigl(\Bm - z^{\star}\Id_N\bigr)^{-1}\;.
\]
We next introduce a large-$N$ equivalent of $\matr\Delta$, denoted $\matr\Delta_N$. Working with $\matr\Delta_N$ in place of $\matr\Delta$ allows us to derive non-asymptotic concentration inequalities that play a key role in the proof.

We first introduce the random variable
\begin{align}
\epsilon_N^\star\doteq z^N_{\Am}(\chi_N) + z^N_{\Bm}(\chi_N) + \chi_N^{-1}
\end{align}
with $z_{\Am}^N(s)$ denoting the functional inverse of the Stieltjes transform
\( s_{\Am}^N(z) \doteq \phi_N\big((\Am - z \Id_N)^{-1}\big) \) for \( z < 0 \). Furthermore, we introduce the event
\begin{align}
\mathcal E_N\doteq \{\vert\epsilon_N^\star\vert\leq\sigma_{\min}(\Bm)\}\;.
\end{align}
The purpose of introducing the event $\mathcal E_N$ is to ensure the existence of the inverse $(\Om^\top\Am\Om + \Bm - \epsilon_N^\star\Id_N)^{-1}$ as  $(\Om^\top\Am\Om  + \Bm - \epsilon_N^\star\Id_N)\geq \sigma_{\min}(\Am)\Id_N$ on $\mathcal E_N$.  We  define
\begin{align}
    \matr\Delta_N \doteq\mathds{1}_{\mathcal E_N} [\left(\Om^\top\Am\Om + \Bm - \epsilon_N^\star\Id_N\right)^{-1} - \big(\Bm - z^N_{\Bm}(\chi_N)\Id_N\big)^{-1}]\;. \label{DeltaN}
\end{align}
Here, $\mathds{1}_{\mathcal E_N}\in\{0,1\}$ denotes the indicator of the event ${\mathcal E_N}$, so that on the complement $\mathcal E_N^{\rm c}$ we have $\matr\Delta_N=\matr 0$.

\begin{lemma}\label{lemma_init}
Under Assumption~\ref{as1}, we verify in Appendix~\ref{App_DD} that
\begin{subequations}
    \label{DD}
    \begin{align}
\lim_{N\to \infty} \left(z_{\Am,\Bm}^N(\chi_N)- z_{\Am,\Bm}(\chi)\right)&\overset{\rm a.s.}{=}0 \\
\lim_{N\to\infty} \epsilon_N^\star&\overset{\rm a.s.}{=}0\label{l2}\\
\lim_{N\to \infty}\sigma_{\max}(\mathds{1}_{\mathcal E_N}\matr\Delta-\matr\Delta_N)&\overset{\rm a.s.}{=}0\;.\label{DDc}
\end{align}
\end{subequations}
\end{lemma}
Note that \eqref{l2} gives
$\mathds{1}_{\mathcal E_N}\overset{\rm a.s.}{\to}1$ as $N\to \infty$. 
Hence, we have the implication
\begin{align}
 \lim_{N\to\infty}(1-\mathds{1}_{\mathcal E_N})\delta_N\overset{\rm a.s.}{=}0 \quad \text{if}~~ \limsup_{N\to \infty}{\delta_N}\overset{\rm a.s.}{<}\infty\;.\label{implication}   
\end{align}
Hence, the $\Deltam_N$ substitution can be used directly for the asymptotic results in
Theorems~\ref{th1} and~\ref{th2}. Indeed, writing
$\Deltam-\Deltam_N=(\mathds 1_{\mathcal E_N}\Deltam-\Deltam_N)+(1-\mathds 1_{\mathcal E_N})\Deltam$
and noting that $\limsup_{N\to\infty}\sigma_{\max}(\Deltam)\overset{\rm a.s.}{<}\infty$ under
Assumption~\ref{as1}, the second term is handled by the implication above; for the first, the
trivial inequality $|\uv^\top\Xm\vv|\leq\sigma_{\max}(\Xm)$, valid for $\|\uv\|=\|\vv\|=1$, together
with Lemma~\ref{lemma_init} gives
\[
    \limsup_{N \to \infty} \left\vert \frac{1}{N^q} \operatorname{tr}\!\left(
        \Km \left[
            \mathds 1_{\mathcal E_N}\matr\Delta-\matr\Delta_N
        \right]
    \right) \right\vert\leq \limsup_{N \to \infty} \max_i\vert\lambda_i\vert\cdot \sigma_{\max}(\mathds{1}_{\mathcal E_N}\matr\Delta-\matr\Delta_N)
\overset{\rm a.s.}{=} 0\;.
\]
\subsection{The Dynamical Mean-Field Approach}
Setting $z^\star_N \doteq z^N_\Bm(\chi_N)$, we write $\gammav$ in \eqref{gammav0}
explicitly as
\begin{align}
\gammav&\doteq \sqrt{N}(\Bm-{z}_{\Bm}^N(\chi_N)\Id_N)\matr \Delta_N(\Bm-{z}_{\Bm}^N(\chi_N)\Id_N)\uv\label{gammav}\;.
\end{align}
Note that $\gammav$ involves the
indicator $\mathds{1}_{\mathcal{E}_N}$ through $\matr\Delta_N$; in particular $\mathds{1}_{\mathcal{E}_N} \gammav = \gammav$.

Starting from an independent random initialization $\gammav^{(0)}=\sqrt{\tau}\widetilde\gv^{(0)}$ with $\widetilde \gv^{(0)}\sim\mathcal N(\matr 0,\Id_N)$, and for iteration steps $t=1,2,\ldots$, we proceed as
\begin{align}
\label{oamp}
\gammav^{(t)}&=\Om^\top\Em_A\Om[\Em_B\gammav^{(t-1)}+\sqrt{N}\chi_N^{-1}\uv]\;.
\end{align}
Here we note that $\gammav^{(t)}$ is always well defined, both on $\mathcal E_N$ and on $\mathcal E_N^{\rm c}$.

We note the following identities:
\begin{align}
\phi_N(\Em_A)&=\phi_N(\Em_B)=0\\
\mathds{1}_{\mathcal E_N}\left(\Om^\top\Am\Om + \Bm - \epsilon_N^\star\Id_N\right)^{-1}\;
&=\mathds{1}_{\mathcal E_N}\chi_N(\Em_B+\Id_N)
(\Id_N-\Om^\top\Em_A\Om\Em_B)^{-1}
(\Om^\top\Em_A\Om+\Id_N)\label{sep}\;.
\end{align}
The first identity follows from the definition of ${z}_{\Am}^N(s)$, which gives $\phi_N((\Am-{z}_{\Am}^N(s)\Id_N)^{-1})=s_{\Am}^N({z}_{\Am}^N(s))=s$. The second identity uses the manipulation \[(\Xm+\Ym-s\Id_N)^{-1}=\Xm^{-1}(\Xm^{-1}+\Ym^{-1}-s\Ym^{-1}\Xm^{-1})^{-1}\Ym^{-1}\;.\]

\begin{lemma}\label{P1}
For any $t\geq 0$, we have
\begin{equation}
\mathds{1}_{\mathcal E_N}(\gammav -\gammav^{(t)})=\mathds{1}_{\mathcal E_N}(\Id_N-\Om^\top\Em_A\Om\Em_B)^{-1}(\gammav^{(t+1)}-\gammav^{(t)})\;.   \label{gammassep}
\end{equation}
\begin{proof}
Note that on the event $\mathcal E_N^{\rm c}$ both
sides of \eqref{gammassep} vanish. Also, for short, we consider the substitution
\[\Om^\top\Em_A\Om\to \Em_{A}\;.\] Conditioned on event  $\mathcal E_N$ we have from~\eqref{sep}
\begin{align}
\gammav= \sqrt{N}[(\Id_N-\Em_A\Em_B)^{-1}(\Em_A+\Id_N)-\Id_N](\Bm-{z}_{\Bm}^N(\chi_N)\Id_N)\uv\;.
\end{align}
Multiplying both sides by $(\Id_N-\Em_A\Em_B)$ gives
\begin{align}
\gammav=\Em_A\Em_B\gammav+\sqrt{N}\Em_A\uv/\chi_N\;.
\end{align}
Then, for the deviations $\matr\delta^{(t)}\doteq \gammav-\gammav^{(t)}$, we obtain the recursion $\matr\delta^{(t+1)}=\Em_A\Em_B\matr\delta^{(t)}$.
Note also that
\begin{align}
\gammav^{(t+1)}-\gammav^{(t)}=\matr\delta^{(t)}-\matr\delta^{(t+1)}=(\Id_N-\Em_A\Em_B)\matr\delta^{(t)}\;.
\end{align}
Multiplying both sides by $(\Id_N-\Em_A\Em_B)^{-1}$ yields \eqref{gammassep}.
\end{proof}
\end{lemma}
Below, we prove Theorems~\ref{th1} and~\ref{th2} in the case where neither
${\rm F}_{\Am}$ (the limiting eigenvalue distribution of $\Am$) nor ${\rm F}_{\Bm}$ is a Dirac measure. The case in which ${\rm F}_{\Am}$ or ${\rm F}_{\Bm}$ is a Dirac measure at a nonzero point is treated separately in Appendix~\ref{DiracFA}.

\subsection{The Non-Asymptotic Analysis}
\subsubsection*{Concentration Inequalities with $\Opn{\kappa}$}
 We analyze the dynamics $\gammav^{(t)}$ using concentration inequalities in terms of
$\mathcal{L}^p$ norms conditioned on the \emph{quenched} sigma-algebra
\begin{equation}
\mathcal{F}_N \doteq \sigma\bigl(\Am,\, \Bm,\, \Om_{\rm new},\uv\bigr)
\label{eq:sigmaalgebra}\;.
\end{equation}
Note that  $\chi_N$, $z^N_{\Am,\Bm}(\chi_N)$, $\phi_N(\Em_{A,B}^2)$ are all $\mathcal{F}_N$-measurable, while $(\Om,\gammav^{(0)})$  is independent of
$\mathcal{F}_N$. The bounds are expressed in terms of the random variable
\begin{equation}
\theta_N \doteq \max\left\{
\frac{1}{\chi_N},\;
\frac{1}{\lvert z^N_{\Am}(\chi_N)\rvert},\;
\frac{1}{\lvert z^N_{\Bm}(\chi_N)\rvert},\;
\frac{1}{\phi_N(\Em_{A}^2)\,\phi_N(\Em_{B}^2)}
\right\}\geq 1
\label{eq:lambdaN}
\end{equation}
where we assume that neither $\Am$ nor $\Bm$ is proportional to $\Id_N$ (so that $\phi_N(\Em_{A,B}^2) > 0$). Above, the inequality is due to the fact that $\chi_N\vert z_{\Am}^N(\chi_N)\vert\leq 1$. 

For a sequence of random variables $X \equiv X_N$, we write $X = \Opn{1}$ if there
exists a deterministic $D \geq 0$ such that for every $p \in \mathbb{N}$ there is a
deterministic finite constant $C_p$, with $C_p$ and $D$ independent of $N$ and of the
realization of the quenched randomness, for which
\begin{equation}
\mathbb{E}\bigl[\lvert X \rvert^p \,\big\vert\, \mathcal{F}_N\bigr]^{1/p}
\leq\theta_N^{D} C_p\;.
\label{eq:opnotation}
\end{equation}
Moreover, for a sequence of random variables $Y \equiv Y_N$, we write
$Y = \Op{1}$ if
\begin{equation}
\mathbb{E}\bigl[\lvert Y \rvert^p \,\big\vert\, \mathcal{F}_N\bigr]^{1/p}
\leq C_p
\label{eq:opuniform}
\end{equation}
for every $p \in \mathbb{N}$ and deterministic constants $C_p$. Directly from the two
definitions, if $X = \Opn{1}$ with exponent $D$ as in \eqref{eq:opnotation}, then
\begin{equation}
Y=\theta_N^{-D} X = \Op{1}\;.
\end{equation}

By a slight abuse of notation, for a random matrix $\Xm \equiv \Xm_N \in \mathbb{R}^{d \times t'}$
(we will use $d = 1$ or $d = N$, with $t'$ independent of $N$) and for any
deterministic $\kappa > 0$ (e.g.\ $\kappa = N$, $\kappa = 1$, or $\kappa = 1/\sqrt{N}$), we
write
\begin{equation}
\Xm = \Opn{\kappa}
\quad \text{if} \quad
\tfrac{1}{\kappa}\lVert \Xm \rVert_{\rm F} = \Opn{1}.
\label{eq:matrixnotation}
\end{equation}

Finally, we record the following elementary arithmetic properties of $\Opn{\kappa}$ which will be frequently used: for any random variables $A_N = \Opn{\kappa}$ and $B_N = \Opn{\tilde \kappa}$
\begin{subequations}
\label{arit}
\begin{align}
A_N+B_N&=\Opn{\max(\kappa,\widetilde \kappa)}, \\
A_NB_N&=\Opn{\kappa\tilde\kappa},\\
\sqrt{1+A_N}&=1+\Opn{\kappa}
\end{align}
\end{subequations}
The first two properties follow from the Minkowski inequality and H\"older's inequality, respectively; the third from the bound $\vert 1- \sqrt{1+x}\vert=\frac{\vert x\vert}{1+\sqrt{1+x}}\leq \vert x\vert$.
\subsubsection*{The Non-Asymptotic Decoupling}
We now recall the $(T'+1)\times (T'+1)$ matrix $\mathcal C_N^{(0:T')}$ introduced
in \eqref{C_N}, which is symmetric positive definite whenever  neither $\Am$ nor $\Bm$ is proportional to $\Id_N$ and $\tau>0$ (which is the case when ${\rm F}_{\Am}$ is not a Dirac measure). Hence, its Cholesky
decomposition is unique, and we denote it as
\begin{align}
\label{BdefTh2}
{\mathcal B}_N^{(0:T')}&={\rm chol}({\mathcal C}_N^{(0:T')})\;,
\end{align}
where ${\mathcal B}_N^{(0:T')}$ is the upper-triangular $(T'+1)\times (T'+1)$ matrix
with positive diagonal entries such that
${\mathcal C}_N^{(0:T')}=({\mathcal B}_N^{(0:T')})^\top{\mathcal B}_N^{(0:T')}$.
We then introduce "effective'' dynamics as
\begin{align}\label{eff_dynamics}
\gammav_{\rm e}^{(t')}= \sum_{0\leq s\leq t'}{{\mathcal B}}^{(s,t')}_N\widetilde\gv^{(s)},
\qquad 0\leq t'\leq T'<N/2;
\end{align}
Here, we recall $\gammav^{(0)}=\sqrt{\tau}\widetilde\gv^{(0)}$ and
$\{\widetilde\gv^{(t+1)}\sim\mathcal {N}(\matr 0,\Id_N)\}_{t}$ is a set of independent Gaussian random vectors, all independent of the quenched sigma-algebra $\mathcal F_{N}$. Later in Section~\ref{ProofP2} the independent Gaussian vectors $\{\widetilde\gv^{(t+1)}\sim\mathcal {N}(\matr 0,\Id_N)\}_{t}$ will be used in representing the Haar matrix $\Om$. 

\begin{proposition}\label{P2}
Let $\Om$ be a Haar random matrix independent of $\mathcal F_N$. 
We verify in the next section that for any $t\geq 0$ fixed with respect to $N$, we have 
\begin{equation}
\gammav^{(t)}=\gammav_{\rm e}^{(t)}+\Opn{1}\label{decom}
 \end{equation}
whenever neither $\Am$ nor $\Bm$ is proportional to $\Id_N$ and $\tau>0$.
\end{proposition}
Now, for any $T>0$ fixed w.r.t. $N$ let us define
\begin{align}
\Gm\doteq [\widetilde\gv^{(0)},\widetilde\gv^{(1)},\cdots, \widetilde\gv^{(T)}]\mathcal B_N^{(0:T)}(\mathcal C_N^{(0:T)})^{-\frac 1 2}
\end{align}
We then have from Proposition~\ref{P2} that 
\begin{align}
\bigl[\gammav^{(0)},\gammav^{(1)},\cdots, \gammav^{(T)}\bigr]= \Gm\sqrt{ \mathcal{C}_{N}^{(0:T)}}+\Opn{1}\label{non_new}
\end{align}
Note that  $\mathcal B_N^{(0:T)}(\mathcal C_N^{(0:T)})^{-\frac 1 2}$ is an orthogonal matrix and thereby $\Gm$ is a standard Gaussian matrix independent of $\mathcal F_N$. This verifies the non-asymptotic decomposition \eqref{explicit} given that $\theta_N$ is almost surely bounded as $N\to\infty$ which we will verify in the sequel.

Second, we underline that $\Opn{1}$ term in \eqref{decom} or \eqref{non_new} obeys $\mathcal L^p$ bounds that are uniform over $\{\uv:\norm{\uv}=1\}$:
\begin{remark}\label{universality}
For any unit-norm vector $\uv$, there exists an orthogonal matrix $\Um$, chosen as an $\mathcal
F_N$-measurable function of $\uv$, such that $\uv=\Um\ev_1$, where $\ev_1$ is a standard basis
vector in $\RR^N$. Since $\Om$ is Haar-distributed and independent of $(\mathcal F_N,\gammav
^{(0)})$, the matrix $\Om\Um$ is Haar-distributed conditionally on $\mathcal F_N$, hence
independent of $(\mathcal F_N,\gammav^{(0)})$. Replacing $\gammav^{(t)}$ by $\Um^\top\gammav^{(
t)}$, $\gammav^{(0)}$ by $\Um^\top\gammav^{(0)}\overset{d}{=}\gammav^{(0)}$, $\Om\Um$ by $\Om$,
and $\Em_B$ by $\Um^\top\Em_B\Um$, we obtain the recursion
\begin{align}
\label{oampe}
\gammav^{(t)}&=\Om^\top\Em_A\Om[\Um^\top\Em_B\Um\gammav^{(t-1)}+\sqrt{N}\chi_N^{-1}\ev_1]\;.
\end{align}
Hence, the decomposition \eqref{decom} is valid. In particular, $\Um\Gm\sim\Gm$, which is independent of $\mathcal F_N$, and $\norm{\Um\Psim}_{\rm F}=\norm{\Psim}_{\rm F}$ for $\Psim=\Opn{1}$. This verifies that $\Opn{1}$ in \eqref{decom}
or \eqref{non_new} gives $\mathcal L^p$ bounds uniform over $\{\uv:\norm{\uv}=1\}$.
\end{remark}

While $\mathcal{C}_{N}^{(0:T)}$ is random through $\phi_N(\Em_A^2)$ and $\phi_N(\Em_B^2)$ and $\chi_N$, it admits the following limiting expression.

\begin{lemma}\label{Cequivalence}
We introduce  $(T'+1)\times (T'+1) $ matrix $\mathcal C^{(0:T')}$ such that for any $0\leq s,t\leq T'$ we set 
\begin{equation}
\mathcal C^{(t,s)}\doteq \left\{\begin{array}{cc}
   \tau   & t=s  \\
   \tau(1-\rho^{\min(t,s)}) & t\neq s
\end{array}\right. 
\end{equation}
where $\tau$ is defined as in \eqref{tau} and 
\begin{equation}
\rho\doteq \frac{{\rm R}_{\Am}'(-\chi)}{\frac{1}{\eta}
    +{\rm R}_{\Am}'(-\chi)} \cdot
    \frac{{\rm R}_{\Bm}'(-\chi)}{\frac{1}{\eta}
    +{\rm R}_{\Bm}'(-\chi)}<1\;.
\end{equation}
Here, the upper bound follows from the fact that $\eta<\infty$ and ${\rm R}'_{\Am,\Bm}(s)\geq 0$.
In fact, we have 
\[\lim_{N\to\infty}\phi_N(\Em_{A}^2)\phi_N(\Em_{B}^2)\overset{\rm a.s.}{=}\rho\;.\] 
Furthermore, for any $s,t\geq 0$ fixed w.r.t. $N$, we have
\begin{equation}
\lim_{N\to \infty}\mathcal{C}_N^{(t,s)}\overset{\rm a.s.}{=}\mathcal{C}^{(t,s)}\;.
\end{equation}
\begin{proof}
 See Appendix~\ref{Proof_Cequivalence}.   
\end{proof}

\end{lemma}

From Lemma~\ref{lemma_init} and Lemma~\ref{Cequivalence} we have that
\begin{align}
\lim_{N\to\infty}\theta_N\overset{\rm a.s.}{=}\underbrace{{\max\left\{\frac{1}{\chi},\frac{1}{\vert z_{\Am}(\chi)\vert},\frac{1}{\vert z_{\Bm}(\chi)\vert},\frac{1}{\rho}\right\}}}_{\doteq \theta}\label{eq:lambda}
\end{align}
when neither ${\rm F}_{\Am}$ (the limiting eigenvalue distribution of $\Am$) nor ${\rm F}_{\Bm}$ is a Dirac measure. 

We are now ready to prove Theorems~\ref {th1}--\ref{th2}. 
\subsection{Proof of Theorem~\ref{th1}}
We note that
\begin{align}
 \left\vert \frac{1}{N^q} {\rm tr}\!\left(
        \Km \matr\Delta_N
        \right)\right\vert\leq  \max_i\vert\lambda_i\vert  \cdot\frac{1}{N^q}\sum_{i\leq\lfloor N^q \rfloor} {\vert \vv_i^\top \matr\Delta_N\uv_i\vert}\;.
\end{align}
Now, for convenience, we fix an arbitrary index $i \leq \lfloor N^q \rfloor$ and write $\mathbf{u'}\equiv \mathbf{u}_i$ and $\mathbf{v'} \equiv \mathbf{v}_i$.
Furthermore,  let
\begin{align}
{\uv}\doteq\frac{(\Bm-{z}_{\Bm}^N(\chi_N)\Id_N)^{-1}\uv'}{\norm{
(\Bm-{z}_{\Bm}^{N}(\chi_N)\Id_N)^{-1}\uv'}} \quad \text{and} \quad
{\vv}\doteq\frac{(\Bm-{z}_{\Bm}^N(\chi_N)\Id_N)^{-1}\vv'}{\norm{
(\Bm-{z}_{\Bm}^{N}(\chi_N)\Id_N)^{-1}\vv'}}\;.
\end{align}

Using the inequality $\norm{(\Bm-{z}_{\Bm}^N(\chi_N)\Id_N)^{-1}\uv'}\leq -1/{{{z}_{\Bm}^N(\chi_N)}}$, we then write
\begin{align}
\vert(\vv')^\top \matr\Delta_N\uv'\vert\leq \frac{1}{{z}_{\Bm}^{N}(\chi_N)^2}\frac{\vert\vv^\top\gammav\vert}{\sqrt N}\;.
\end{align}
From Lemma~\ref{P1}, we have for every $t\geq 0$
\begin{align}
\frac{\vert\vv^\top\gammav\vert}{\sqrt N}&\leq \frac{\vert\vv^\top\gammav^{(t)}\vert}{\sqrt N}+ \sigma_{\max}\!\left(\mathds{1}_{\mathcal E_N}\big(\Id_N-\Om^\top\Em_A\Om\Em_B\big)^{-1}\right)\frac{\norm{\gammav^{(t+1)}-\gammav^{(t)}}}{\sqrt N}\;.
\label{bound1}
\end{align}
Furthermore, from $\eqref{sep}$ it follows that 
\begin{align}
\sigma_{\max}\!\left(\mathds{1}_{\mathcal E_N}\big(\Id_N-\Om\top\Em_A\Om\Em_B\big)^{-1}\right)\leq \chi_N\frac{(\sigma_{\max}(\Am)+1/\chi_N)(\sigma_{\max}(\Bm)+1/\chi_N)}{\sigma_{\min}(\Am)}\label{nbounEAB}\;.
\end{align}
Hence, we trivially have
\begin{align}
\limsup_{N\to\infty} \sigma_{\max}\!\left(\mathds{1}_{\mathcal E_N}\big(\Id_N-\Om^\top\Em_A\Om\Em_B\big)^{-1}\right)\overset{\rm a.s.}{<}\infty.\label{bounEAB}
\end{align}

From Proposition~\ref{P2}, for $t\geq 0$ fixed w.r.t. $N$, we have
\begin{align}
\frac{\vert\vv^\top\gammav^{(t)}\vert}{\sqrt N}&=\frac{\vert\vv^\top\gammav_{\rm e}^{(t)}\vert}{\sqrt N}+\Opn{N^{-\frac 1 2}}\\
&\overset{(a)}{=}\sqrt{\mathcal C_N^{(t,t)}}\frac{\vert Z\vert }{\sqrt N}+\Opn{N^{-\frac 1 2}}=\Opn{N^{-\frac 1 2}}\\
\frac{\norm{\gammav^{(t+1)}-\gammav^{(t)}}}{\sqrt N}&=\frac{\norm{\gammav_{\rm e}^{(t+1)}-\gammav_{\rm e}^{(t)}}}{\sqrt N}+\Opn{N^{-\frac 1 2}}\\
&\overset{(b)}{=}\sqrt{\mathcal C_{N}^{(t,t)}+\mathcal C_N^{(t+1,t+1)}-2\mathcal C_{N}^{(t+1,t)}+\Opn{N^{-\frac 1 2}}}+\Opn{N^{-\frac 1 2}}\\
&\overset{(c)}{=}\underbrace{\sqrt{\mathcal C_{N}^{(t,t)}+\mathcal C_N^{(t+1,t+1)}-2\mathcal C_{N}^{(t+1,t)}}}_{\overset{\rm a.s.}{\rightarrow} \sqrt{2\tau\rho^{t}}}+\Opn{N^{-\frac 1 2}}\; \label{2t}
\end{align}
As to step (a), we note that since $\Om$ is independent of $(\Theta,\vv)$, the
Haar-free representation applies jointly with $\vv$,
with the reference Gaussian inputs $\{\widetilde\gv^{(t)}\}_{t}$ chosen independently
of $(\Theta,\vv)$. Then, from $\norm{\vv}=1$ we have $\vv^\top\gammav_{\rm e}^{(t)}=\sqrt{C_N^{(t,t)}}Z$ where  $Z\sim\mathcal N(0,1)$. Step $(b)$ applies
Lemma~\ref{Gaussiian-Con} with $\Xm=\Id_N$. Step $(c)$ follows from the arithmetic
properties of $\Opn{\kappa}$ in \eqref{arit}, and we have the bound (see~\eqref{minC})
\begin{equation}
\mathcal C_{N}^{(t,t)}+\mathcal C_N^{(t+1,t+1)}-2\mathcal C_{N}^{(t+1,t)}
\geq \tau\min(1,(\phi_N(\Em_A^2)\phi_N(\Em_B^2))^{t})\;,
\end{equation}
whose reciprocal is bounded by $\theta_N^{t}/\tau=\Opn{1}$. 

Now, we first note that above $\Opn{N^{-\frac 1 2}}$ refers implicitly to a random variable $X_i$ such that \[\mathbb E[\vert X_i\vert^p\mid \mathcal F_N]^{\frac 1 p}\leq N^{-\frac 1 2} \theta_N^DC_p\] 
where $C_p,D>0$ are constants independent of the indices $i\le\lfloor N^{q}\rfloor$: uniformity in the right-hand
vectors $\uv_i$ follows from Remark~\ref{universality}, while the left-hand vectors
$\vv_i$ require no separate argument at all, as
$|\vv_i^\top \deltav| \le \|\vv_i\| \, \|\deltav\| = \|\deltav\|$ for every
$\deltav \in \mathbb{R}^N$.

Furthermore, if $X_i=\Opn{N^{-\frac 1 2}}$, then there is a constant $D>0$ such that 
\begin{align}
Y_i\doteq\theta_N^{-D}X_i=\Op{N^{-\frac 1 2}}\;.
\end{align}
We also note from \eqref{eq:lambda} that 
\begin{align}
\lim_{N\to\infty}\theta_N^{D}\overset{\rm a.s.}{=}\theta^{D}.
\end{align}
Moreover, we record the following elementary auxiliary result: if $Y_i=\Op{N^{-1/2}}$ for each $i\leq \lfloor N^q \rfloor$, then for any fixed $p$ we have
\begin{align}
\PP\left(\frac{1}{\lfloor N^q \rfloor}\sum_{i\leq \lfloor N^q \rfloor} \vert Y_i\vert\geq \epsilon\right)&\leq \sum_{i\leq \lfloor N^q \rfloor}\PP\!\left( \vert Y_i\vert\geq \epsilon\right)
\leq N^q \frac{C_{p}^{p}}{\epsilon^{p}}N^{-\frac{p}{2}}\;. \label{pp1}
\end{align}
Choosing $p>2q+3$ and applying the Borel--Cantelli lemma then gives
\begin{align}
\frac{1}{N^q}\sum_{i\leq \lfloor N^q \rfloor} \vert Y_i\vert\overset{\rm a.s.}{\longrightarrow} 0 \quad \text{as} \quad N\to \infty.\label{pp2}
\end{align}

Putting everything together, there exists an almost surely finite random variable $L'$ such that
\begin{align}
\limsup_{N\to \infty}\left\vert \frac{1}{N^q} {\rm tr}\!\left(
        \Km \matr\Delta_N
        \right)\right\vert\leq L'\rho^{\frac{t}{2}}\;.
\end{align}
Here $L’$ can be chosen independently of $t$.
The bound holds simultaneously for all integer $t\geq0$
on a probability-one event. Letting $t\to\infty$ and using
$\rho<1$ proves the claim.
\subsection{Proof of Theorem~\ref{th2}}
First, it follows readily from Lemma~\ref{lemma_init} that
\begin{align}
\lim_{N\to \infty}\frac{1}{\sqrt  N}\left\Vert \sqrt{N}\,
(\Bm-z^\star\Id_N)
\left[
(\Om^\top \Am \Om + \Bm)^{-1}
-
(\Bm-z^\star\Id_N)^{-1}
\right]
(\Bm-z^\star\Id_N)\uv-  \gammav\right\Vert\overset{\rm a.s.}{=}0\;.
\end{align}
Then, in the sequel, we verify that 
\begin{align}
\lim_{N\to \infty}\frac{\norm{\gammav-\sqrt{\tau}\gv_N}}{\sqrt N}\overset{\rm a.s.}{=}0\;.
\end{align}
To this end, analogous to the effective dynamics $\gammav_e^{(t')}$ in \eqref{eff_dynamics} we define 
\begin{align}
\gammav_{\infty}^{(t')}&= \sum_{0\leq s\leq t'}{{\mathcal B}}^{(s,t')}\widetilde\gv^{(s)},
\qquad 0\leq t'\leq T'<N/2,
\end{align}
which is driven by the same Gaussian vectors, while the coefficients ${\mathcal B}^{(t,s)}$ are defined via the Cholesky decomposition (see Lemma~\ref{Cequivalence})
\[ {\mathcal B}^{(0:T')}={\rm chol}({\mathcal C}^{(0:T')})\;. \]
For convenience, we set $\gammav^\star_{\infty}\doteq \gammav_{\infty}^{(\lfloor N/2\rfloor-1)}$
and note that $\gammav^\star_{\infty}\sim \mathcal N(\matr 0,\tau\Id_N)$. We then define $\gv_N
\doteq\gammav_\infty^\star/\sqrt\tau$, which is independent of $\Theta$, and in particular of $
(\Am,\Bm,\uv)$.

Let $t$ be fixed with respect to $N$. Then, we first write
\begin{align}\label{chain}
\norm{\gammav-\gammav_{\infty}^\star}\leq\norm{\gammav-\gammav_{\infty}^{(t)}}+\norm{\gammav_{\infty}^\star-\gammav_{\infty}^{(t)}}\leq  \norm{\gammav-\mathds{1}_{\mathcal E_N}\gammav^{(t)}}+\norm{\gammav_{\infty}^{(t)}-\mathds{1}_{\mathcal E_N}\gammav^{(t)}}+
\norm{\gammav_{\infty}^\star-\gammav_{\infty}^{(t)}}\;.
\end{align}
We first consider the last term in \eqref{chain}. A direct application of Lemma~\ref{Cequivalence} yields
\begin{equation}
\lim_{N\to \infty}\frac{\norm{\gammav_{\infty}^\star-\gammav_{\infty}^{(t)}}}{\sqrt N}\overset{\rm a.s.}{=}\sqrt{2\tau}\rho^{\frac t 2}\;.
\end{equation}
Next, consider the first term in \eqref{chain}. Applying Lemma~\ref{P1} and using
$\mathds{1}_{\mathcal E_N}\gammav=\gammav$, gives
\begin{align}
\frac{\norm{\gammav-\mathds{1}_{\mathcal E_N}\gammav^{(t)}}}{\sqrt N}\leq
\sigma_{\max}\!\left(\mathds{1}_{\mathcal E_N}(\Id_N-\Om^\top\Em_A\Om\Em_B)^{-1}\right)
\frac{\norm{\gammav^{(t+1)}-\gammav^{(t)}}}{\sqrt N}\;.
\end{align}
Furthermore, we recall \eqref{2t}, i.e., 
\begin{align*}
\frac{\norm{\gammav^{(t+1)}-\gammav^{(t)}}}{\sqrt N}=\underbrace{\sqrt{\mathcal C_{N}^{(t,t)}+\mathcal C_N^{(t+1,t+1)}-2\mathcal C_{N}^{(t+1,t)}}}_{\overset{\rm a.s.}{\rightarrow} \sqrt{2\tau}\rho^{\frac t 2}}+\Opn{N^{-\frac 1 2}}\;.
\end{align*}
Note also that if $\delta_N=\Opn{N^{-c}}$ for some constant $c>0$, then $\lim_{N\to\infty}\delta_N\overset{\rm a.s.}{=}0$.
Indeed, writing $\delta_N=\theta_N^{D}\widetilde\delta_N$ with $\widetilde\delta_N=\Op{N^{-c}}$, Markov's inequality gives
$\PP\bigl(\vert\widetilde\delta_N\vert\geq \epsilon\bigr)\leq \frac{C_p^{p}}{\epsilon^{p}}N^{-c p}$ for every
$p\in\NN$ and $\epsilon>0$. Choosing $p>1/c$ gives $\widetilde\delta_N\overset{\rm a.s.}{\to}0$ by the
Borel--Cantelli lemma; the claim follows since $\lim_{N\to\infty}\theta_N^{D}\overset{\rm a.s.}{=}\theta^{D}<\infty$.
Consequently, we have
\begin{align}
\limsup_{N\to \infty}\frac{\norm{\gammav-\mathds{1}_{\mathcal E_N}\gammav^{(t)}}}{\sqrt N}\overset{\rm a.s.}{\leq} L\sqrt{2\tau}\rho^{\frac t 2}\;,\label{boundth2}
\end{align}
where 
$L\doteq\limsup_{N\to\infty}\sigma_{\max}\!\left(\mathds{1}_{\mathcal E_N}(\Id_N-\Om^\top\Em_A\Om\Em_B)^{-1}\right)$
is almost surely finite (see \eqref{bounEAB}). 
It remains to control the middle term in \eqref{chain}.  We recall the property of $\mathds{1}_{\mathcal E_N}$ in \eqref{implication}. Moreover, 
Lemma~\ref{Cequivalence} gives $\mathcal C_N^{(0:t)}\overset{\rm a.s.}{\to}\mathcal C^{(0:t)}$ and letting $N\to\infty$ in \eqref{minC} gives
$\mathcal C^{(0:t)}\geq\tau\rho^{t}\Id_{t+1}$, hence $\mathcal B_N^{(0:t)}\overset{\rm a.s.}{\to}\mathcal B^{(0:t)}$, we obtain

\begin{align}
\lim_{N\to \infty}\frac{\norm{\gammav_{\infty}^{(t)}-\mathds{1}_{\mathcal E_N}\gammav^{(t)}}}{\sqrt N}\overset{\rm a.s.}{=}0\;.
\end{align}

Putting everything together, we get
\begin{equation}
\limsup_{N\to\infty}\frac{\norm{\gammav-\sqrt\tau\,\gv_N}}{\sqrt N}\overset{\rm a.s.}{\leq}(L+1)\sqrt{2\tau}\rho^{\frac t 2}\;.
\end{equation}
Since $t$ can be taken arbitrarily large and $\rho<1$, the claim follows by letting $t\to\infty$.
\section{Proof of Proposition~\ref{P2}}\label{ProofP2}
We rewrite the dynamics \eqref{oamp} as
\begin{subequations}
\label{new_dyn}
\begin{align}
\widetilde\gammav^{(t)}&=\Em_B\gammav^{(t-1)}+\sqrt{N}\uv/\chi_{N}\\
\qv^{(t)}&=\Om\widetilde\gammav^{(t)}\\
\psiv^{(t)}&=\Em_{A}\qv^{(t)}\\
\gammav^{(t)}&=\Om^\top\psiv^{(t)}\;.
\end{align}
\end{subequations}
The dynamics \eqref{new_dyn} are coupled by the Haar (quenched-disorder) matrix $\Om$ via the products $\{\Om\widetilde\gammav^{(t)},\Om^\top\psiv^{(t)}\}$.
As a first step, we follow the method of the ``Householder Dice'' representation of Haar-coupled iterative equations  \cite[Section~III-C]{lu2021householder} to construct a statistically equivalent Haar-free (i.e., $\Om$-free) version of the dynamics in \eqref{new_dyn}.
\subsection{The Haar-Free Equivalent}
We use the classical Gram--Schmidt notation. Let $\vv^{(1:t-1)}= [\vv^{(1)},\vv^{(2)},\ldots, \vv^{(t-1)}]$
be a collection of orthogonal vectors in $\RR^{N}$ such that $\langle \vv^{(s)},\vv^{(s')} \rangle = \delta_{ss'}$ for all $s,s'$, where $\langle\av,\bv\rangle\doteq\frac1N\av^\top\bv$. Denote the projection onto the complement of $\text{span}(\vv^{(1:t-1)})$ by
\begin{equation}
\Pm^\perp_{\vv^{(1:t-1)}}\doteq \Id_N-\frac 1 N\vv^{(1:t-1)}(\vv^{(1:t-1)})^\top=\Id_N-\frac 1 N\sum_{1\leq s<t}\vv^{(s)}(\vv^{(s)})^\top \;.
\end{equation}
For any vector $\bv\in \RR^{N}$ we construct the orthogonal vector $\vv^{(t)}=\GS{\bv}{\vv^{(1:t-1)}}$ where
\begin{equation}
\GS{\bv}{\vv^{(1:t-1)}} \doteq \sqrt N\frac{\Pm^\perp_{\vv^{(1:t-1)}}\bv}{\Vert \Pm^\perp_{\vv^{(1:t-1)}}\bv \Vert}
\end{equation}
unless  $\bv\in{\rm span}(\vv^{(1:t-1)})$.  If $\bv\in{\rm span}(\vv^{(1:t-1)})$,  we instead generate an independent Gaussian vector $\bv_{\rm new}\sim\mathcal N(\matr 0,\Id_N)$ and set
$
\vv^{(t)}=\GS{\bv_{\rm new}}{\vv^{(1:t-1)}}.
$
We also recall the Gram--Schmidt decomposition
\begin{equation}\label{gsd}
\bv=\sum_{1\leq s\leq t}\vv^{(s)}\langle\vv^{(s)},\bv \rangle\;.
\end{equation}

\begin{lemma}\label{lemmaseq} 
Let $N>2T'$. Let $\Om$ be a Haar matrix independent of $\Theta\doteq(\Am,\Bm,\Om_{\rm new},\uv,\gammav^{(0)})$.
Then, the joint distribution of $\Theta$ and
$\gammav^{(0:T')}\equiv
[{{\gammav}^{(0)},{\gammav}^{(1)},\cdots,{\gammav}^{(T')}}]$ generated by the dynamics \eqref{oamp} from $\gammav^{(0)}$, is the same as the joint distribution of $\Theta$ and the sequence generated by the
following dynamics: starting from $\gammav^{(0)}$, for $t=1,2,\cdots,T'$ we construct
\begin{subequations}
\label{hauseholderep}
\begin{align}
\widetilde\gammav^{(t)}&=\Em_B\gammav^{(t-1)}+\sqrt N\uv/\chi_N\\
\vv^{(2t-1)}&=\mathcal{GS}(\gv^{(t)}\vert \vv^{(1:2(t-1))})\\
\widetilde\vv^{(2t-1)}&=\mathcal{GS}({\widetilde\gammav}^{(t)}\vert \widetilde\vv^{(1:2(t-1))})\\
\qv^{(t)}&=\sum_{1\leq s< 2t}\vv^{(s)}\langle\widetilde\vv^{(s)},{\widetilde\gammav}^{(t)} \rangle \\
\psiv^{(t)}&=\Em_{A}\qv^{(t)}\\
\vv^{(2t)}&=\mathcal{GS}({\psiv}^{(t)}\vert \vv^{(1:2t-1)})\\
\widetilde\vv^{(2t)}&=\mathcal{GS}(\widetilde{\gv}^{(t)}\vert \widetilde\vv^{(1:2t-1)})\\
\gammav^{(t)}&=\sum_{1\leq s\leq 2t}\widetilde\vv^{(s)}\langle\vv^{(s)},{\psiv}^{(t)} \rangle.
\end{align}
\end{subequations}
Here, $[\gv^{(1:T')},\widetilde\gv^{(1:T')}]\in\RR^{N\times 2T'}$ and the auxiliary vectors
$\{\bv_{\rm new}\}$ possibly used by $\mathcal{GS}$ consist of independent standard Gaussian entries, jointly
independent of $\Theta$.
\begin{proof}
The proof of Lemma~\ref{lemmaseq} is deferred to Appendix~\ref{Proof_HD}.
\end{proof}
\end{lemma}
Fix $N\geq2$ and let $T'\doteq\lfloor N/2\rfloor-1$. Recall that $\Theta=(\Am,\Bm,\Om_{\rm new},\uv,\gammav^{(0)})$,
and write $\xi\doteq\gammav^{(0:T')}$ for the sequence generated by \eqref{oamp}. Collect the reference Gaussian
inputs into
\[
\tilde\eta
\doteq
\bigl(
\{\gv^{(t)},\widetilde\gv^{(t)}\}_{1\leq t\leq T'},
\{\bv_{\rm new}^{(k)}\}_{1\leq k\leq2T'}
\bigr),
\]
where $\bv_{\rm new}^{(k)}$ is used only in the $k$-th data-driven call of $\mathcal{GS}$ in \eqref{hauseholderep},
i.e., the calls with input $\widetilde\gammav^{(t)}$ or $\psiv^{(t)}$. The components of $\tilde\eta$ are mutually
independent standard Gaussian vectors, and $\tilde\eta$ is independent of $\Theta$. Let $F(\tilde\eta,\Theta)$ denote
the sequence $\gammav^{(0:T')}$ generated by \eqref{hauseholderep}. Then, by Lemma~\ref{lemmaseq},
\begin{equation}
(\xi,\Theta)
\overset d=
(F(\tilde\eta,\Theta),\Theta).\label{suf}
\end{equation}
This equality in law suffices for the proof of
Theorem~\ref{th1}. For Theorem~\ref{th2}, however, we need
to realize the Gaussian inputs on an enlargement of the
original probability space so that the original and
Haar-free trajectories coincide almost surely.
The following remark provides this coupling.
\begin{remark}\label{transfer}
We first make the measurability of $F$ explicit. The Gram--Schmidt operation $\mathcal{GS}(\bv\vert\vv^{(1:t-1)})$
is Borel measurable jointly in $(\bv,\vv^{(1:t-1)},\bv_{\rm new})$ on the domain
\[
\bigl\{\Pm^\perp_{\vv^{(1:t-1)}}\bv\neq\matr0\bigr\}
\;\cup\;
\bigl\{\Pm^\perp_{\vv^{(1:t-1)}}\bv=\matr0,\ \Pm^\perp_{\vv^{(1:t-1)}}\bv_{\rm new}\neq\matr0\bigr\},
\]
since these are Borel pieces and the defining formula is continuous on each piece. Hence $F$, a finite composition
of continuous operations and of $\mathcal{GS}$, is Borel measurable on a Borel domain $D$. Moreover,
\[
\PP\bigl((\tilde\eta,\Theta)\in D\bigr)=1.
\]
To see this, note that each Gaussian input of $\mathcal{GS}$, i.e., a fresh $\gv^{(t)}$ or $\widetilde\gv^{(t)}$, or
$\bv_{\rm new}^{(k)}$ by its assignment above, is independent of the frame constructed before its call, whose span
has dimension at most $2T'-1<N$; hence its projection onto the orthogonal complement of this span is nonzero almost
surely. We extend $F$ to the complement of $D$ by setting $F(e,z)=\matr0$ there, so that $F$ is Borel measurable
everywhere.

By \eqref{suf} and the transfer theorem \cite[Theorem~8.17]{kallenberg1997foundations}, there exists, on an
enlargement of the underlying probability space, a random element $\eta$ such that
\begin{equation}
(\xi,\Theta,\eta)
\overset d=
(F(\tilde\eta,\Theta),\Theta,\tilde\eta).
\label{eq:transfer-coupling}
\end{equation}
More precisely, $\eta$ may be constructed as $\eta=H(\xi,\Theta,U)$ for a Borel measurable function $H$ and a
uniform random variable $U$ on $[0,1]$ that is independent of the original probability space\cite{kallenberg1997foundations}. In particular,
\eqref{eq:transfer-coupling} gives
\[
(\Theta,\eta)
\overset d=
(\Theta,\tilde\eta),
\]
so the components of $\eta$ are mutually independent standard Gaussian vectors, and $\eta$ is independent of
$\Theta$.

Since $F$ is Borel measurable, the set $G\doteq\{(x,z,e):x=F(e,z)\}$ is Borel, and the right-hand side of
\eqref{eq:transfer-coupling} belongs to $G$. Consequently, $\PP\bigl((\xi,\Theta,\eta)\in G\bigr)=1$, i.e.,
\[
\xi\overset{\rm a.s.}{=}F(\eta,\Theta)\;.
\]
We henceforth use this coupling and denote the components of $\eta$ again by $\gv^{(1:T')}$,
$\widetilde\gv^{(1:T')}$, and $\bv_{\rm new}^{(k)}$. Thus, the original iterates of \eqref{oamp} coincide almost
surely with those of \eqref{hauseholderep} driven by this Gaussian family, which is independent of $\Theta$ but is
not asserted to be independent of $\Om$. 

We can carry out the constructions for all $N\geq2$ together
by adding one uniform random variable $U_N$ on $[0,1]$
for each $N$. These additional variables are mutually
independent and jointly independent of all the original
random variables. Since $N$ ranges over the integers $2,3,4,\ldots$, the
probability that any of the coupling identities fails
is bounded by a countable sum of zero probabilities.
Thus, with probability one, the identity holds for every \(N\geq 2\) simultaneously.

All original random variables, including their dependence
across different $N$, remain unchanged. Although the added
variables $U_N$ are independent, the resulting Gaussian
families for different $N$ need not be independent.
\end{remark}

\subsection{The High-Dimensional Equivalent}\label{step2}

At a high level, we complete the proof of Proposition~\ref{P2} as follows: in the first step (Section~\ref{is}), we apply a relatively simple concentration result, which yields a convenient high-dimensional equivalent—up to $\Opn{1}$ deviations—of the Haar-free equivalent dynamics \eqref{hauseholderep}. In the second step (Section~\ref{theo-dynm}), we explicitly define the final effective high-dimensional dynamics and establish several useful concentration results for it. In the final step (Section~\ref{pbi}), we prove inductively (over the iteration index) that the initially obtained dynamics is equivalent, up to $\Opn{1}$ deviations, to the effective dynamics.

\subsubsection{Initial Simplifications}\label{is}

\begin{lemma}\label{laux1}
For any $t\leq T$ we have
\label{aux1}
\begin{align}
	\widetilde\gammav^{(t)}=\Opn{\sqrt N}\quad \text{and}\quad
\psiv^{(t)}=\Opn{\sqrt N}\;.
	\end{align}
\begin{proof}
First, we note that 
\begin{equation}
\sigma_{\max}(\Em_{A,B})\leq \frac{1}{\chi_N\vert z_{\Am,\Bm}^N(\chi_N)\vert}+1=\Opn{1}. \label{sigmaeab}
\end{equation}
Also, $\gammav^{(0)}=\Op{\sqrt N}$  since $\gammav^{(0)}\sim\mathcal{N}(\matr 0,\tau\Id_N)$.  The claim then follows by induction over the iterations using the arithmetic properties of $\Opn{\kappa}$ in \eqref{arit} applied to \eqref{new_dyn}.
\end{proof}
\end{lemma}

In particular, these results imply that for any $s,t\leq T$,
\begin{subequations}
 \label{aux11}
\begin{align}
\vert\langle \widetilde\vv^{(s)},\widetilde\gammav^{(t)}\rangle\vert&\overset{(a)}{\leq}  \frac{1}{N}\norm{\widetilde\vv^{(s)}}\norm{\widetilde\gammav^{(t)}}\nonumber \\
&=\frac{1}{\sqrt N}\norm{\widetilde\gammav^{(t)}}=\Opn{1},\\
\vert\langle\vv^{(s)},\psiv^{(t)}\rangle\vert&\overset{(b)}{\leq}\frac{1}{N}\norm{\vv^{(s)}}\norm{\psiv^{(t)}}\nonumber \\
&=\frac{1}{\sqrt N}\norm{\psiv^{(t)}}=\Opn{1},
\end{align}
\end{subequations}
where steps (a) and (b) follow from the Cauchy--Schwarz inequality.

\begin{lemma}\label{laux2}
Let $\Pm\in\RR^{N\times N}$ be a possibly random
orthogonal projection of rank $t'$, where $t'$ is fixed
with respect to $N$ and $N>t'$.
Let $\gv\sim\mathcal N(\matr0,\Id_N)$ be independent
of $\sigma(\mathcal F_N,\Pm)$. Then,
\begin{equation}
\vv\doteq
\sqrt{N}\frac{(\Id_N-\Pm)\gv}
{\norm{(\Id_N-\Pm)\gv}}
=\gv+\Op{1}.
\end{equation}
\begin{proof}
Set $\mathcal G_N\doteq\sigma(\mathcal F_N,\Pm)$.
Consider the eigenvalue decomposition $\Pm=\Um^\top {\rm diag}(\matr 1_{t'},\matr 0_{N-t'})\Um$. Condition on $\mathcal G_N$, we have $\gv_{new}\doteq\Um\gv\sim \mathcal N(\matr 0,\Id_N)$ and it is independent of $\mathcal G_N$. Hence, we get
\begin{align}
\norm{\Pm\gv} &=\norm{\gv'}=\Op{1}\\
g_{\perp}\doteq \frac{1}{\sqrt{N}}\norm{(\Id_N-\Pm)\gv}&=\frac{1}{\sqrt{N}}\norm{\gv''}\\
&=\sqrt{1+\Op{N^{-\frac 1 2}}}=1+\Op{N^{-\frac 1 2}}
\end{align}
where $\gv'\sim\mathcal N(\matr0,\Id_{t'})$
and $\gv''\sim\mathcal N(\matr0,\Id_{N-t'})$ and the pair $(\gv',\gv'')$
is independent of $\mathcal G_N$. Hence, 
\begin{align}
\norm{\vv-\gv}^2=N(1-g_{\perp})^2+\norm{\Pm\gv}^2= \Op{1}.
\end{align}
\end{proof}
\end{lemma}
From Lemma~\ref{laux2}, we have for any $t\in[T]$:
\begin{subequations}
 \label{aux22}
\begin{align}
			\vv^{(2t-1)}&= \gv^{(t)}+\Op{1},\\
			\widetilde\vv^{(2t)}&=\widetilde\gv^{(t)}+\Op{1}\;.
		\end{align}
\end{subequations}

Combining \eqref{aux11} and \eqref{aux22} with the arithmetic properties of $\Opn{\kappa}$ in \eqref{arit}, we obtain:
\begin{subequations}
\begin{align}
{\qv}^{(t)}&=
\sum_{1\leq s\leq t}\widehat{{\widetilde{\mathcal B}}}_N^{(s,t)}
\gv^{(s)} +\matr\delta_{q}^{(t)}+\Opn{1}\label{qv},\\
\gammav^{(t)}&=\sum_{1\leq s\leq t}\widehat{{\mathcal B}}_N^{(s,t)}\widetilde\gv^{(s)}+\matr \delta_{\gamma}^{(t)}+\Opn{1}\label{phiv},
\end{align}
where, for all $s\leq t$,
\begin{align}
		\widehat{\widetilde{{\mathcal B}}}_N^{(s,t)}&\doteq \langle \widetilde\vv^{(2s-1)},\widetilde\gammav^{(t)} \rangle,\label{b1}\\
		\widehat{{{\mathcal B}}}^{(s,t)}_N&\doteq \langle \vv^{(2s)},\psiv^{(t)}\rangle,\label{b2}\\
		\matr \delta_{q}^{(t)}&\doteq\sum_{1\leq s< t}\vv^{(2s)}\langle \widetilde\gv^{(s)},\widetilde{\gammav}^{(t)} \rangle, \\
		\matr \delta_{\gamma}^{(t)}&\doteq \sum_{1\leq s\leq t}\widetilde\vv^{(2s-1)}\langle\gv^{(s)},\psiv^{(t)} \rangle.
	\end{align}
\end{subequations}

It is convenient to express the empirical order parameters $\widehat{\widetilde{{\mathcal B}}}_N^{(s,t)}$ and $\widehat{{{\mathcal B}}}_N^{(s,t)}$ via Cholesky decompositions. To this end, we introduce the projected dynamics
\begin{subequations}
 \begin{align}
\widetilde\gammav_\perp^{(t)}&\doteq \left(\Id_N-\frac{1}{N}\sum_{1\leq s<t}\widetilde\vv^{(2s)}(\widetilde\vv^{(2s)})^\top\right)\widetilde\gammav^{(t)},\\
\psiv_\perp^{(t)}&\doteq  \left(\Id_N-\frac{1}{N}\sum_{1\leq s\leq t}\vv^{(2s-1)}(\vv^{(2s-1)})^\top\right)\psiv^{(t)}\;,
\end{align}
\end{subequations}
and define the empirical cross-correlations for $t,s\in[T]$:
\begin{subequations}
 \begin{align}
\widehat{\widetilde{\mathcal C}}_N^{(t,s)}&\doteq
 \langle \widetilde {\gammav}_\perp^{(t)}, \widetilde {\gammav}_\perp^{(s)} \rangle, \\
 \widehat{\mathcal C}_N^{(t,s)}&\doteq\langle{\psiv}_\perp^{(t)},{\psiv}_\perp^{(s)} \rangle \;.
\end{align}
\end{subequations}
In particular, from \eqref{gsd} notice that
\begin{subequations}
 \begin{align}
\widehat{\widetilde{\mathcal B}}_N^{(1:T)}&={\rm chol}(\widehat{\widetilde{\mathcal C}}_N^{(1:T)}),\\
\widehat{\mathcal B}_N^{(1:T)}&={\rm chol}(\widehat{\mathcal C}_N^{(1:T)})\;.\label{ebg2}
\end{align}
\end{subequations}
Here, we recall the \emph{Cholesky decompositions} notation: For any $\mathcal{A}^{(1:t')} \geq \mathbf{0}$, let $\mathcal{B}^{(1:t')}$ be the corresponding $t' \times t'$ upper-triangular factor such that $\mathcal{A}^{(1:t')}
= \big(\mathcal{B}^{(1:t')}\big)^{\top} \mathcal{B}^{(1:t')}$, and write $\mathcal B^{(1:t')}={\rm chol}(\mathcal A^{(1:t')})$.
\subsubsection{The Effective Dynamics}\label{theo-dynm}
 By construction, we have 
\begin{equation}\label{eq:Crec2}
  \mathcal C^{(t+1,s+1)}_N=a_{N}+\rho_N \mathcal C^{(t,s)}_N
  \quad \text{with}~~ \mathcal C^{(0,t)}_N=\tau\delta_{0,t}
\end{equation}
where for short we define $\rho_N \doteq \phi_N(\Em_A^2)\phi_N(\Em_B^2)$ and
$a_N \doteq \phi_N(\Em_A^2)/\chi_N^2$. 
Thus, for any $t,s$,
\begin{equation}\label{eq:Cclosed}
  \mathcal C^{(t,s)}_N \;=\; \tau\rho_N^{\,t}\delta_{ts}
  \;+\; a_N\sum_{j=0}^{\min(t,s)-1}\rho_N^{\,j}\,.
\end{equation}
We then express the second term of \eqref{eq:Cclosed} as the $(t,s)$ entry of the $(T'+1)\times (T'+1)$ matrix
\[
  a_N\sum_{j=0}^{T'-1}\rho_N^{\,j}\mathds{1}_{j<t} \mathds{1}_{j<s}
  \;=a_N\;\left(\sum_{j=0}^{T'-1}\rho_N^{\,j}\,x_jx_j^{\top}\right)_{t,s},
  \quad (x_j)_t\doteq\mathds{1}_{j<t}\;.
\]
 Consequently we have, 
\begin{equation}  
\mathcal C^{(0:T')}_N\;\geq \;\tau\,\diag\big(1,\rho_N,\dots,\rho_N^{\,T'}\big)\geq \tau\min\{1,\rho_N^{T'}\}\Id_{T'+1}. \label{minC}
\end{equation}
\begin{definition} \label{effec_dyn}
We recall the independent Gaussian random vectors $\gv^{(s)}$ and $\widetilde\gv^{(s)}$ in Lemma~\ref{lemmaseq}. Then, starting from $\gammav_{e}^{(0)}\equiv \gammav^{(0)}=\sqrt{\tau}\widetilde\gv^{(0)}$, we construct for $t=1,2,\cdots,T'<N/2$
\begin{subequations}
\begin{align}
\widetilde\gammav_{e}^{(t)}&=\Em_B\gammav_e^{(t-1)}+\sqrt{N}\uv/\chi_N\\
{\qv}_e^{(t)}&=
\sum_{1\leq s\leq t}{{\widetilde{\mathcal B}}}_N^{(s,t)}
\gv^{(s)}\label{qvv}\\
\psiv_{e}^{(t)}&=\Em_A{\qv}_e^{(t)}\\
\gammav_e^{(t)}&= \sum_{1\leq s\leq t}{{\mathcal B}}^{(s,t)}_N\widetilde\gv^{(s)}\;,
\end{align}
\end{subequations}
where the coefficients ${\mathcal B}_N^{(t,s)}$ and $\widetilde{\mathcal B}_N^{(t,s)}$ are defined via the Cholesky decompositions
\begin{subequations}
\label{Bdef}
 \begin{align}
\widetilde{\mathcal B}_N^{(1:T')}&={\rm chol}(\widetilde{\mathcal C}_N^{(1:T')}),\\
\mathcal B_N^{(1:T')}&={\rm chol}(\mathcal C_N^{(1:T')})\;.
\end{align}
\end{subequations}
Here, the matrix $\mathcal C_N^{(0:T')}$ is as constructed in \eqref{C_N}, and we define
\begin{align}
\widetilde{\mathcal C}^{(1:T')}_N\doteq\phi_N(\Em_B^2)\mathcal C_N^{(0:T'-1)}+ \frac{1}{\chi_N^2}11^\top\;,
\end{align}
where $1$ denotes the $T'$-dimensional all-ones vector. Note from \eqref{minC} that the Cholesky decompositions are well defined, and
they are unique provided that neither $\Am$ nor $\Bm$ is proportional to $\Id_N$ and $\tau>0$. 
\end{definition}

Here we note that the two indexings of $\mathcal B_N$ used in this paper are
consistent. Indeed, since $\mathcal C_N^{(0,t)}=\tau\delta_{0t}$, the matrix
$\mathcal C_N^{(0:T')}$ is block diagonal with blocks $\tau$ and
$\mathcal C_N^{(1:T')}$, and hence so is its Cholesky factor:
\[
\mathcal B_N^{(0,0)}=\sqrt\tau,\qquad
\mathcal B_N^{(0,t)}=0 \quad (1\leq t\leq T'),\qquad
\bigl(\mathcal B_N^{(s,t)}\bigr)_{1\leq s,t\leq T'}={\rm chol}(\mathcal C_N^{(1:T')})\;.
\]
Thus the entries $\mathcal B_N^{(s,t)}$ with $1\leq s\leq t$ are unambiguous,
and the effective dynamics defined here agrees with \eqref{eff_dynamics}: the
summand $s=0$ there contributes only at $t'=0$, where it gives
$\gammav_{\rm e}^{(0)}=\sqrt\tau\,\widetilde\gv^{(0)}=\gammav^{(0)}$.

While the effective dynamics is well-defined for any $t<N/2$, in the proof of Proposition~\ref{P2} we only refer to it for $t\leq T$ where $T$ is fixed w.r.t. $N$. In particular, we have in this case
\begin{align}
\label{chp}
\frac{1}{\sigma_{\min}(\mathcal C_{N}^{(1:T)})}&=\Opn{1}\\
\frac{1}{\sigma_{\min}(\widetilde{\mathcal C}_{N}^{(1:T)})}&=\Opn{1}\;.
\end{align}

Next, we verify some useful concentration results for the effective dynamics. We first present the following auxiliary result:
\begin{lemma}\label{Gaussiian-Con}
Let $\Xm=\Xm^\top\in\mathbb{R}^{N\times N}$,
$\mathcal C_N\in\mathbb{R}^{2\times 2}$, and
$\xv\in\mathbb{R}^N$ be $\mathcal F_N$-measurable.
Assume that $\mathcal C_N$ is positive semidefinite and that
\[
\mathcal C_N=\Opn{1},\qquad
\sigma_{\max}(\Xm)=\Opn{1},\qquad
\norm{\xv}=\Opn{\sqrt N}.
\]
Let $\zv_1,\zv_2\sim\mathcal N(\mathbf{0},\Id_N)$
be independent standard Gaussian vectors such that the pair
$(\zv_1,\zv_2)$ is independent of $\mathcal F_N$, and define
\[
[\gammav_1,\gammav_2]
\doteq [\zv_1,\zv_2]\sqrt{\mathcal C_N}.
\]
Then, for all $i,j\in\{1,2\}$,
\begin{align}
\langle\gammav_1,\xv\rangle
&=\Opn{N^{-\frac12}},\\
\langle\gammav_i,\Xm\gammav_j\rangle
&=(\mathcal C_N)_{ij}\phi_N(\Xm)
  +\Opn{N^{-\frac12}}.
\end{align}
\begin{proof}
We prove the second statement only; the proof of the first is analogous and simpler.
Consider the eigenvalue decomposition $\Xm=\Um^\top{\rm diag}(\av)\Um$. By the rotational invariance $(\gammav_1,\gammav_2)\sim (\Um\gammav_1,\Um\gammav_2)$, we may assume without loss of generality that $\Xm={\rm diag}(\av)$. Writing $\mathcal P_1\doteq [1, 0]$ and $\mathcal P_2\doteq [0,1]$, we have
\begin{align}
\langle\gammav_i,\Xm\gammav_j\rangle= \mathcal P_{i}\sqrt{\mathcal C}_N\frac{[\zv_1, \zv_2]^\top{\rm diag}(\av)~[\zv_1, \zv_2]}{N}\sqrt{\mathcal C}_N\mathcal P_j^\top.
\end{align}
For given $i,j\in\{1,2\}$, define
\begin{equation}
  S_N\doteq\langle\zv_i,{\rm diag}(\av)\zv_j\rangle- \phi_N(\Xm)\delta_{ij}=\frac{1}{N}\sum_{n\leq N}a_nS_n,
\end{equation}
where $S_n\equiv z_{ni}z_{nj}-\delta_{ij}$. By the product rule of $\Opn{\kappa}$, we have $S_n=\Op{1}$. 
Conditionally on $\mathcal F_N$, the Gaussian coordinates used
in $S_N$ are independent. Hence using \cite[Lemma 7.8]{erdHos2017dynamical}, we have
\[
\mathbb E[|S_N|^p\mid\mathcal F_N]^{1/p}
\leq C_p\frac{\|\av\|}{N}.
\]
Using $\|\av\|\leq\sqrt N\,\sigma_{\max}(\Xm)$ gives
$S_N=\Opn{N^{-1/2}}$. In other words,
\begin{equation}
\left(\frac{[\zv_1, \zv_2]^\top{\rm diag}(\av)~[\zv_1, \zv_2]}{N}\right)_{ij}=\phi_{N}(\Xm)\delta_{ij}+\Opn{N^{-\frac 1 2}}\;.
\end{equation}
Recalling that $\mathcal C_N=\Opn{1}$, the claim follows from the product rule of $\Opn{\kappa}$.
\end{proof}
\end{lemma}

We recall that $\phi_N(\Em_B)=\phi_N( \Em_{A})=0$. Then, it follows readily from Lemma~\ref{Gaussiian-Con} that
\begin{subequations}
\label{concen}
\begin{align}
\langle\widetilde\gv^{(s)},\widetilde \gammav_e^{(t)}\rangle&= \Opn{N^{-\frac 1 2}},  \label{mem1}\\
\langle\widetilde\gammav_{e}^{(t)},\widetilde\gammav_{e}^{(s)} \rangle&=\widetilde{\mathcal C}_N^{(t,s)} +\Opn{N^{-\frac 1 2}}, \label{Ctilde}\\
\langle\gv^{(s)},\psiv_{e}^{(t)}\rangle&= \Opn{N^{-\frac 1 2}}, \label{mem2}\\
\langle\psiv_{e}^{(t)},\psiv_{e}^{(s)} \rangle&={\mathcal C}_N^{(t,s)}+\Opn{N^{-\frac 1 2}} \;.\label{C}
\end{align}
\end{subequations}

\subsubsection{Proof by Induction}\label{pbi}
Let ${\mathcal H}_{t}$ denote the hypothesis that
\begin{subequations}
 \begin{align}
 \widehat{\widetilde{\mathcal B}}_N^{(1:t)}&=
\widetilde{\mathcal B}_N^{(1:t)}+\Opn{N^{-\frac 1 2}},\quad
\matr \delta_{q}^{(1:t)}= \Opn{1},\label{H21}\\
 \widehat {{\mathcal B}}_N^{(1:t)}&={\mathcal B}_N^{(1:t)}+ \Opn{N^{-\frac 1 2}},\quad \matr \delta_{\gamma}^{(1:t)}=\Opn{1}\;.  \label{H22}
	\end{align}
\end{subequations}
We prove the result by verifying the induction step ${\mathcal H}_{t-1} \implies {\mathcal H}_{t}$ and the base case ${\mathcal H}_1$. We first address an induction-step result from which the remainder follows easily.

\begin{lemma}
\label{lemchol1}
Let $\widehat{\mathcal C}\geq\matr 0,\mathcal C>\matr 0$ be $T\times T$ matrices with $\widehat{\mathcal C}=\Opn{1}$ and $1/\sigma_{\min}(\mathcal C)=\Opn{1}$. Suppose $\widehat{\mathcal C}-\mathcal C=\Opn{N^{-c}}$ for some constant $c>0$. Note that ${\rm chol}(\widehat{\mathcal C})$ may not be unique since $\widehat{\mathcal C}$ may be singular; for any such factor, we have the concentration ${\rm chol}(\widehat{\mathcal C})-{\rm chol}({\mathcal C})=\Opn{N^{-c}}$\;.
\begin{proof}
The proof follows from the arithmetic properties of $\Opn{\kappa}$ in~\eqref{arit} applied recursively to the system equations of the Cholesky decompositions for $\mathcal {\widehat C}$ and $\mathcal C$: for $\mathcal B = \operatorname{chol}(\mathcal C)$, the entries of $\mathcal B$ satisfy
\[
\mathcal B^{(s,t)} \mathcal B^{(s,s)}
=
\mathcal C^{(t,s)}
-
\sum_{s'=1}^{s-1}
\mathcal B^{(s',t)} \mathcal B^{(s',s)}\;,
\]
and $\widehat{\mathcal B}$ satisfies the same equations with $\mathcal C$ replaced by $\widehat{\mathcal C}$. In particular, we note that
$|\mathcal B^{(s,t)}|\leq(\mathcal C^{(t,t)})^{1/2}=\Opn{1}$ and $|\widehat{\mathcal B}^{(s,t)}|\leq(\widehat{\mathcal C}^{(t,t)})^{1/2}=\Opn{1}$,
and that $(\mathcal B^{(s,s)})^{2}\geq\sigma_{\min}(\mathcal C)$, so that $1/\mathcal B^{(s,s)}=\Opn{1}$ by the assumption
$1/\sigma_{\min}(\mathcal C)=\Opn{1}$.
\end{proof}
 \end{lemma}
We note from \eqref{aux11} that $\widehat{\mathcal C}_N^{(1:T)}=\Opn{1}$ and $\widehat{\widetilde{\mathcal C}}_N^{(1:t)}=\Opn{1}$.
Applying Lemma~\ref{lemchol1}, we obtain for any $t\leq T$:
\begin{subequations}
\label{chol}
\begin{align}
\widehat{\widetilde{\mathcal C}}_N^{(1:t)}-\widetilde {\mathcal  C}_N^{(1:t)}=\Opn{N^{-\frac 1 2}}&\implies \widehat{\widetilde{\mathcal B}}_N^{(1:t)}-\widetilde {\mathcal B}_N^{(1:t)}=\Opn{N^{-\frac 1 2}}\;,\label{chol2}\\
\widehat{\mathcal C}_N^{(1:t)}-\mathcal C_N^{(1:t)}= \Opn{N^{-\frac 1 2}} &\implies\widehat{\mathcal B}_N^{(1:t)}-\mathcal B^{(1:t)}_N=\Opn{N^{-\frac 1 2}}\;.\label{chol1}
\end{align}
\end{subequations}

Now, let us verify the induction step ${\mathcal H}_{t-1}\implies {\mathcal H}_{t}$ for $t>1$: From  $\gammav^{(t-1)}=\gammav_e^{(t-1)}+\Opn{1}$ and $\sigma_{\max}(\Em_B)=\Opn{1}$, we have
\begin{equation}
\widetilde\gammav^{(t)}=\widetilde\gammav_{e}^{(t)}+\Opn{1}.
\end{equation}
Hence, from \eqref{mem1}, for any $1\leq s\leq t$,
\begin{align}
\langle\widetilde\gv^{(s)},\widetilde \gammav^{(t)}\rangle&=\Opn{N^{-\frac 1 2}} \;.\label{mem1r}
\end{align}
This implies $\deltav_{q}^{(t)}= \Opn{1}$.
Moreover, \eqref{mem1r} and \eqref{aux22} together imply $\langle\widetilde\vv^{(2s)},\widetilde \gammav^{(t)}\rangle=\Opn{N^{-1/2}}$. Thus,
\begin{align}
\widetilde\gammav_\perp^{(t)}= \widetilde\gammav^{(t)}+\Opn{1}= \widetilde\gammav_{e}^{(t)}+\Opn{1}\;.
\end{align}
Applying \eqref{Ctilde} then gives
\begin{align}
\widehat{\widetilde{\mathcal C}}_N^{(1:t)}- {\widetilde{\mathcal C}_N^{(1:t)}}=\Opn{N^{-\frac 1 2}} \;,
\end{align}
and from \eqref{chol2}, we have
\begin{align}
\widehat{\widetilde{\mathcal B}}_N^{(1:t)}=\widetilde{\mathcal B}^{(1:t)}_N+\Opn{N^{-\frac 1 2}} \;.
\end{align}
This verifies that $\mathcal H_{t-1}$ implies \eqref{H21} for $t$.

Now, given that \eqref{H21} holds for all $s\leq t$, from $\sigma_{\max}(\Em_A)=\Opn{1}$ and $\qv^{(t)}=\qv_e^{(t)}+\Opn{1}$, we have
\begin{equation}
\psiv^{(t)}= \psiv_{e}^{(t)}+\Opn{1}.
\end{equation}
Hence, from \eqref{mem2}, for any $1\leq s\leq t$,
\begin{align}
\langle\gv^{(s)},\psiv^{(t)}\rangle&= \Opn{N^{-\frac 1 2}}\;,\label{mem2r}
\end{align}
which implies $\deltav_{\gamma}^{(t)}=\Opn{1}$.
Moreover, \eqref{mem2r} and \eqref{aux22} give $\langle\vv^{(2s-1)},\psiv^{(t)}\rangle=\Opn{N^{-1/2}}$, so that
\begin{align}
\psiv_{\perp}^{(t)}= \psiv^{(t)}+\Opn{1}=\psiv_{e}^{(t)}+\Opn{1}\;.
\end{align}
Applying \eqref{C} then gives
\begin{align}
\widehat{{\mathcal C}}^{(1:t)}_N- {{\mathcal C}^{(1:t)}_N}= \Opn{N^{-\frac 1 2}}\;,
\end{align}
and from \eqref{chol1} we have
\begin{align}
\widehat{{\mathcal B}}_N^{(1:t)}={\mathcal B}_N^{(1:t)}+\Opn{N^{-\frac 1 2}}\;.
\end{align}
This completes the induction step $\mathcal H_{t-1}\implies \mathcal H_{t}$.

For the base case $\mathcal{H}_1$, we note that $\gammav_e^{(0)}\equiv \gammav^{(0)}$, and thus
\begin{align}
 \widehat{\widetilde{\mathcal{B}}}_N^{(1,1)} = \frac{\norm{\widetilde\gammav^{(1)}}}{\sqrt N}=\sqrt{\widetilde{\mathcal C}_N^{(1,1)}+\Opn{N^{-\frac 1 2}}}={\widetilde{\mathcal{B}}}_N^{(1,1)}+\Opn{N^{-\frac 1 2}}.\label{H11}
\end{align}
Hence \eqref{H21} holds for $t=1$. Since we have already shown that \eqref{H21} for $t>0$ implies \eqref{H22} for $t$, the proof of Proposition~\ref{P2} is complete.
\section{Outlook}\label{outlook}
An extension of our work would consider (higher-order) deterministic equivalents for
matrices involving products of multiple copies of the inverse
$(\Om^\top \Am \Om + \Bm)^{-1}$ (see e.g.\ \cite{couillet2022random}), the simplest
case being $(\Om^\top \Am \Om + \Bm)^{-2}$. We expect that such a generalization is
possible by considering the joint statistics of multiple vectors of the type
\eqref{gammav0} with the same random Haar matrix but different vectors $\uv$.

{A further natural extension is to consider the resolvent of the additive model
$(\Om^\top\Am\Om+\Bm-z\Id_N)^{-1}$ at a spectral parameter $z$ in the complex upper
half-plane. We have deliberately not pursued this here, as our aim has been to
transfer the dynamical mean-field approach in a simpler (yet relevant) setting. We
expect that the essential difficulty in the complex-valued case lies in the
stability criterion $\rho<1$, where
\begin{align}
\rho\overset{\rm a.s.}{=} \lim_{N\to \infty}\frac{1}{N}\norm{\Em_{A}}_{\rm F}^2\,
\frac{1}{N}\norm{\Em_{B}}_{\rm F}^2 ,
\end{align}
which drives the geometric convergence of the dynamics and which we expect to hold
for sufficiently small $\Im z$, but not uniformly in $z$.

It will be important to see whether our dynamical approach can establish
deterministic equivalents under more universal assumptions on the random matrices.
This would not only be of mathematical interest but might also be relevant for
applications of random matrix theory to signal processing and machine learning. A
promising direction would be to replace the Haar ensemble by semi-random matrices
such as randomly signed Hadamard matrices
\cite{anderson2014asymptotically,au2021freeness}. This is motivated by recent
progress in analyzing the dynamics of message passing algorithms involving random
matrices, which has revealed a greater universality
\cite{wang2024universality,dudeja2023universality} of results.
\appendix
\section {Proof of Lemma~\ref{lemma_init}}\label{App_DD}
Since the empirical eigenvalue distribution of \( \mathbf A \) converges almost surely to a limiting distribution with compact support as \( N \to \infty \), it follows that \( s_{\mathbf A}^N(z) \overset{\rm a.s.}{\to} s_{\mathbf A}(z) \) for all \( z < 0 \). While pointwise convergence of the Stieltjes transform does not directly imply pointwise convergence of its inverse \( z_{\mathbf A}^N(s) \), we note that
\begin{align}
\frac{\partial s_{\mathbf A}^N(z)}{\partial z}=
\phi_N((\Am-z\Id_N)^{-2})\geq  (s_{\mathbf A}^N(z))^2
\end{align}
Then, by the mean-value theorem, for $\tilde z_N=(1-c_N)z_{\Am}^N(s)+c_Nz_{\Am}(s)$ for some $c_N\in[0,1]$, we have
\begin{align}
\vert z_{\Am}^N(s)-z_{\Am}(s)\vert&\leq \frac{1}{(s_{\mathbf A}^N(\tilde z_N))^2}\vert s-s_{\Am}^N(z_{\Am}(s))\vert\\\
&\leq  \frac{1}{(s_{\mathbf A}^N(-1/s))^2}\vert s-s_{\Am}^N(z_{\Am}(s))\vert\;,
\end{align}
where the second inequality uses $\vert {\rm z}^N_{\Am}(s)\vert,\vert{\rm z}_{\Am}(s)\vert \leq 1/s$ and the monotonicity property
\begin{equation}
s_{\mathbf A}^N(z)<s_{\mathbf A}^N(z')\quad\text{for}\quad  z<z'<0\;.
\end{equation} 
Note also that
\[\lim_{N\to\infty}s_{\Am}^N(z_{\Am}(s))\overset{\rm a.s.}{=}s\;.\]
Thus, we have
\begin{subequations}
\label{coninverse}
\begin{align}
z_{\mathbf A}(s) &\overset{\rm a.s.}{=} \lim_{N \to \infty} z_{\mathbf A}^{N}(s), \quad 0<s< \mu_{\Am^{-1}},
\\
z_{\mathbf B}(s) &\overset{\rm a.s.}{=}\lim_{N \to \infty} z_{\mathbf B}^{N}(s),\quad 0<s<\mu_{\Bm^{-1}},
\end{align}
\end{subequations}
where, e.g., $\mu_{\Am^{-1}}\doteq\int x^{-1}{\rm dF}_{\Am}(x)$.

We next note the bound 
\begin{align}
\vert z_{\Am}^N(\chi_N)-z_{\Am}(\chi)\vert&\leq\vert z_{\Am}^N(\chi)-z_{\Am}(\chi)\vert +\vert z_{\Am}^N(\chi_N)-z_{\Am}^N(\chi)\vert\\
&\leq \vert z_{\Am}^N(\chi)-z_{\Am}(\chi)\vert+\frac{\vert\chi_N-\chi\vert}{\min(\chi,\chi_N)^2}\label{invbound}
\end{align}
where the second step follows from the mean-value theorem together with the bound 
\begin{align}
\frac{\partial z_{\Am}^N(s)}{\partial s}&=\frac{1}{\phi_N((\Am-z_{\Am}^N(s)\Id_N)^{-2})}\\
&\leq \frac{1}{(s_{\mathbf A}^N(z_{\Am}^N(s)))^2}=\frac{1}{s^2}\;.
\end{align}

We recall that $\lim_{N\to \infty}\chi_N\overset{\rm a.s.}{=}\chi$ and
$\chi<\min(\mu_{\Am^{-1}},\mu_{\Bm^{-1}})$. Then, from \eqref{coninverse} and \eqref{invbound} we have
  \begin{align}\label{coninversechi}
\lim_{N \to \infty} z_{\Am,\Bm}^{N}(\chi_N)\overset{\rm a.s.}{=}z_{\Am,\Bm}(\chi)\;.
\end{align}
Hence, we have verified that 
\begin{align}
\lim_{N\to \infty}\epsilon_N^\star\overset{\rm a.s.}{=}0\;.
\end{align}

Now, for $\Cm=\Om^\top\Am\Om+\Bm$ we have
\begin{align}
    \mathds{1}_{\mathcal E_N}(\mathbf{C} -\epsilon_N^\star \Id_N)^{-1} - \mathds{1}_{\mathcal E_N}\mathbf{C}^{-1} = \mathds{1}_{\mathcal E_N}\epsilon_N^\star \mathbf{C}^{-1}(\mathbf{C} -\epsilon_N^\star \Id_N)^{-1}.
\end{align}
Hence, we have
\begin{align}
    \limsup_{N \to \infty} \sigma_{\max}(\mathds{1}_{\mathcal E_N}(\mathbf{C} - \epsilon_N^\star \Id_N)^{-1} - \mathds{1}_{\mathcal E_N}\mathbf{C}^{-1}) \leq \limsup_{N \to \infty}\frac{\vert\epsilon_N^\star\vert}{\sigma_{\min}(\Am)^2} \overset{\rm a.s.}{=} 0.\label{okay}
\end{align}
The proof of \eqref{DDc} is completed by applying an analogous argument to the second term involving $\Bm$.
\section{Proof of Lemma~\ref{Cequivalence}}\label{Proof_Cequivalence}
For short, we denote the limiting normalized trace of a matrix $\Am\in \RR^{N\times N}$ by
\begin{align}
\phi(\Am)\doteq \lim_{N\to\infty}\phi_N(\Am)
\end{align}
whenever the limit exists almost surely. Furthermore, we define the auxiliary matrices
\begin{align}
    \widetilde{\Em}_A &\doteq \frac{1}{\chi}(\Am - z_{\Am}(\chi)\Id_N)^{-1} - \Id_N,\\
     \widetilde{\Em}_B &\doteq \frac{1}{\chi}(\Bm - z_{\Bm}(\chi)\Id_N)^{-1} - \Id_N.
\end{align}
We proceed in two steps.
First, we verify that for any $s\leq t$,
\begin{align}
\mathcal{C}^{(t+1,s+1)} &\overset{\rm a.s.}{=} \phi(\widetilde{\Em}_A ^2)\phi(\widetilde{\Em}_B^2)\mathcal{C}^{(t,s)} + \phi(\widetilde{\Em}_A^2)/\chi^2.
\end{align}
Second, we establish the limiting equivalences
\begin{align}
\phi({\Em}_A^2)&\overset{\rm a.s.}{=}\phi(\widetilde{\Em}_A ^2),\label{step21}\\
\phi({\Em}_B ^2)&\overset{\rm a.s.}{=}\phi(\widetilde{\Em}_B^2)\label{step22}.
\end{align}
Together, these imply the claim.

For the first step, we use the following auxiliary result.
\begin{lemma}\label{iliked}
Consider a random variable $X>0$ with $\mathbb E[X^{-2}]<\infty$. Then,
\begin{equation}
\mathbb E [X^{-2}]=\frac{\mathbb E[X^{-1}]^2}{1-\mathbb E[X^{-1}]^2{\rm R}_{X}'(-\mathbb E[X^{-1}])}
\end{equation}
where ${\rm R}_{X}'$ denotes the derivative of the R-transform of the distribution of $X$.
\begin{proof}
Starting from the identity ${\rm R}_X(-s_X(z))=z+1/s_X(z)$ and differentiating both sides w.r.t. $z$ yields
\begin{equation}
s_X'(z)
=
\frac{s_X(z)^2}
{1 - s_X(z)^2 {\rm R}_X'(-s_X(z))},\quad  z<0.
\end{equation}
Since $X>0$, monotone convergence gives
\begin{align}
\left(\mathbb{E}[X^{-1}],\, \mathbb{E}[X^{-2}]\right)
=
\lim_{z \to 0^-}
\left(s_X(z),\, s_X'(z)\right).
\end{align}
The identity ${\rm R}_X'(-s_X(z))=s_X(z)^{-2}-s_X'(z)^{-1}$ then shows that
${\rm R}_X'(-\mathbb E[X^{-1}])\doteq\lim_{\omega\to-\mathbb E[X^{-1}]^+}{\rm R}_X'(\omega)$ exists and equals
$\mathbb E[X^{-1}]^{-2}-\mathbb E[X^{-2}]^{-1}$, which is the claim.
\end{proof}
\end{lemma}
Applying Lemma~\ref{iliked}, we obtain
\begin{subequations}
\label{nice}
\begin{align}
\eta&= \frac{\chi^2}{1-\chi^2{\rm R}_{\Om^\top\Am\Om+\Bm}'(-\chi)},\\
\phi(\widetilde\Em_A^2)&\overset{\rm a.s.}{=} \frac{\chi^2{\rm R}_{\Am}'(-\chi)}{1-\chi^2{\rm R}_{\Am}'(-\chi)},\\
\phi(\widetilde\Em_B^2)&\overset{\rm a.s.}{=}  \frac{\chi^2{\rm R}_{\Bm}'(-\chi)}{1-\chi^2{\rm R}_{\Bm}'(-\chi)}\;.
\end{align}
\end{subequations}
Since $\Om^\top\Am\Om$ and $\Bm$ are asymptotically free, the additivity of the R-transform gives
\begin{equation}
{\rm R}'_{\Om^\top\Am\Om+\Bm}(-\chi)={\rm R}'_{\Am}(-\chi)+{\rm R}'_{\Bm}(-\chi)\;,
\end{equation}
where the left-hand side is the one-sided limit of Lemma~\ref{iliked}, and the right-hand side is
continuous at $-\chi$ since $-\chi$ is interior to the domains of ${\rm R}_{\Am}$ and ${\rm R}_
{\Bm}$. Hence, the first identity in \eqref{nice} reads
\begin{equation}\label{etaid}
\frac{1}{\eta}=\frac{1}{\chi^2}-{\rm R}'_{\Am}(-\chi)-{\rm R}'_{\Bm}(-\chi)\;.
\end{equation}
Writing for short ${\rm R}'_{\Am}\equiv{\rm R}'_{\Am}(-\chi)$ and
${\rm R}'_{\Bm}\equiv{\rm R}'_{\Bm}(-\chi)$, we obtain from \eqref{etaid} that
$\frac{1}{\eta}+{\rm R}'_{\Am}=\frac{1}{\chi^2}-{\rm R}'_{\Bm}$, and hence
\begin{equation}\label{crossed}
\frac{{\rm R}'_{\Am}}{\frac 1\eta+{\rm R}'_{\Am}}
=\frac{\chi^2{\rm R}'_{\Am}}{1-\chi^2{\rm R}'_{\Bm}}\;,
\qquad
\frac{{\rm R}'_{\Bm}}{\frac 1\eta+{\rm R}'_{\Bm}}
=\frac{\chi^2{\rm R}'_{\Bm}}{1-\chi^2{\rm R}'_{\Am}}\;.
\end{equation}
Multiplying the two identities in \eqref{crossed} and comparing with \eqref{nice}
gives
\begin{equation}
\rho\overset{\rm a.s.}{=}\phi(\widetilde\Em_A^2)\,\phi(\widetilde\Em_B^2)\;,
\end{equation}
Similarly, it follows that 
\begin{align}
\tau&\overset{\rm a.s.}{\equiv} \frac 1{\chi^2}\frac{\phi(\widetilde\Em_A^2)}{1-\rho}\;.
\end{align}
This verifies the first step.

For the second step, we recall that $\phi_N(\Em_A)=0$ (and $\phi_N(\Em_B)=0$). Hence,
\begin{equation}
    \phi_N(\Em_A^2) = \frac{1}{\chi_N^2} \phi_N\left((\Am - z_{\Am}^{N}(\chi_N)\Id_N)^{-2}\right) - 1.
\end{equation}
Let $d_i$ denote the eigenvalues of $(\Am - z_{\Am}(\chi)\Id_N)$ and set $\epsilon_N \doteq z_{\Am}(\chi) - z^N_{\Am}(\chi_N)$. From \eqref{coninversechi} we note that $\epsilon_N \overset{\rm a.s.}{\rightarrow} 0$ as $N\to \infty$. Then,
\begin{align}
    |(d_i + \epsilon_N)^{-2} - d_i^{-2}| = \frac{|\epsilon_N(\epsilon_N + 2d_i)|}{d_i^2(d_i + \epsilon_N)^2}.
\end{align}
Since $\Am>\matr 0$ and ${z}_{\Am}(\chi)<0$, we have $d_i\geq -{z}_{\Am}(\chi)$, which gives the uniform lower bound $d_i+\epsilon_N \geq \vert{\rm z}^N_{\Am}(\chi_N)\vert$. Hence,
\begin{align}
\vert (d_i+\epsilon_N)^{-2}- (d_i)^{-2}\vert\leq \frac{\vert\epsilon_N\vert[\vert\epsilon_N\vert+2(\sigma_{\max}(\Am)+
\vert{z}_{\Am}(\chi)\vert)]}{{ z}_{\Am}(\chi)^2{z}^N_{\Am}(\chi_N)^2}.
\end{align}
This uniform bound yields \eqref{step21}. The proof of \eqref{step22} follows by an identical argument.

\section{Proof of Lemma~\ref{lemmaseq}}\label{Proof_HD}
We use the following inductive rule for constructing Haar orthogonal matrices.
\begin{lemma}[\cite{meckes2019random}\cite{lu2021householder}]\label{conditioning}
Let $\vv^{(1)}=\sqrt{N}\gv^{(1)}/\norm{\gv^{(1)}}$ with
$\gv^{(1)}\sim\mathcal {N}(\matr 0,\Id_N)$, and let a random vector $\widetilde\vv^{(1)}\in \RR^{N}$ with $\norm{\widetilde\vv^{(1)}}=\sqrt{N}$
and a Haar matrix ${\Om_{N-1}}\in \RR^{(N-1)\times (N-1)}$ be mutually independent.
        Then,
		\begin{align}
			\Om&\doteq\frac{1}{N}\vv^{(1)}(\widetilde\vv^{(1)})^\top+\matr\Pi_{\vv^{(1)}}^\perp\Om_{N-1}(\matr\Pi_{\widetilde\vv^{(1)}}^\perp)^\top\label{haarrep}
		\end{align}
		is Haar-distributed and independent of $\widetilde\vv^{(1)}$. Here, $\matr\Pi_{\widetilde\vv^{(1)}}^\perp, \matr\Pi_{\vv^{(1)}}^\perp\in \RR^{N\times(N-1)}$ are semi-orthogonal matrices whose columns span the orthogonal complements of $\widetilde\vv^{(1)}$ and $\vv^{(1)}$, respectively; for example,
		\[\matr\Pi_{\vv^{(1)}}^\perp(\matr\Pi_{\vv^{(1)}}^\perp)^\top=\Id-\frac{1}{N}\vv^{(1)} (\vv^{(1)})^\top=\Pm_{\vv^{(1)}}^\perp.\]
	\end{lemma}
Throughout this appendix, the normalization in $\mathcal{GS}$
is always $\sqrt N$, where $N$ is the dimension of the original
matrix, even when the argument lies in a residual space
$\mathbb R^d$ with $d<N$. In particular, for a nonzero
$\xv\in\mathbb R^d$,
\[
\mathcal{GS}(\xv)=\sqrt N\,\xv/\|\xv\|.
\]
The corresponding rank-one Haar decomposition in dimension $d$
uses $N^{-1}\av\bv^\top$ when both $\av$ and $\bv$ have norm
$\sqrt N$. This is the usual unit-vector decomposition written
with rescaled vectors.

We begin with the iteration step $t=1$. Applying Lemma~\ref{conditioning} with
$\widetilde\vv^{(1)}\equiv\mathcal{GS}(\widetilde\gammav^{(1)})$, we obtain
\begin{equation}
\Om \widetilde\gammav^{(1)}=\langle\widetilde\vv^{(1)},\widetilde\gammav^{(1)} \rangle \vv^{(1)}.
\end{equation}
For the multiplication $\Om^\top \psiv^{(1)}$, we introduce an independent Gaussian vector $\widetilde\gv^{(1)}\sim\mathcal {N}(\matr 0,\Id_N)$. Note that 
	\[(\matr\Pi_{\widetilde\vv^{(1)}}^\perp)^\top\widetilde\gv^{(1)} \sim \mathcal {N}(\matr 0,\Id_{N-1})\]
	and is independent of $\widetilde\vv^{(1)}$. Let $\vv^{(2)}\equiv \GS{\psiv^{(1)}}{\vv^{(1)}}$. Then, we use Lemma~\ref{conditioning} to represent
$\Om_{N-1}$ in \eqref{haarrep} as
	\begin{align}
		\Om_{N-1}=\frac{1}{N}\mathcal{GS}((\matr\Pi_{\vv^{(1)}}^\perp)^\top \vv^{(2)})\mathcal{GS}((\matr\Pi_{\widetilde\vv^{(1)}}^\perp)^\top\widetilde\gv^{(1)})^\top +\matr\Pi_{(\matr\Pi_{\vv^{(1)}}^\perp)^\top \vv^{(2)}}^\perp
		\Om_{N-2}\left(\matr\Pi_{(\matr\Pi_{\widetilde\vv^{(1)}}^\perp)^\top\widetilde\gv^{(1)}}^\perp \right)^\top,
	\end{align}
	where $\Om_{N-2}\in \RR^{(N-2)\times (N-2)}$ is Haar orthogonal.
	Moreover, observe that
	\begin{align}
		\matr\Pi_{\vv^{(1)}}^\perp\mathcal{GS}((\matr\Pi_{\vv^{(1)}}^\perp)^\top\vv^{(2)})&\equiv \vv^{(2)},\label{gset}\\
		\matr\Pi_{\widetilde\vv^{(1)}}^\perp\mathcal{GS}((\matr\Pi_{\widetilde\vv^{(1)}}^\perp)^\top\widetilde\gv^{(1)})
		&=\underbrace{\GS{\widetilde\gv^{(1)}}{\widetilde\vv^{(1)}}}_{\doteq \widetilde \vv^{(2)} }\;.
	\end{align}
	Hence, we have the representation
	\begin{align}\label{usethis}
		\Om=\frac{1}{N} \vv^{(1)}(\widetilde\vv^{(1)})^\top  +\frac{1}{N}\vv^{(2)}(\widetilde\vv^{(2)})^{\top}+\underbrace{\left(\matr\Pi_{\vv^{(1)}}^\perp\matr\Pi_{(\matr\Pi_{\vv^{(1)}}^\perp)^\top\vv^{(2)}}^\perp \right)}_{\doteq \matr\Pi^\perp_{\vv^{(1:2)}}}
		\Om_{N-2}{\underbrace{\left(\matr\Pi_{\widetilde\vv^{(1)}}^\perp\matr\Pi_{(\matr\Pi_{\widetilde\vv^{(1)}}^\perp)^\top\widetilde\gv^{(1)}}^\perp \right)}_{\doteq \matr\Pi^\perp_{\widetilde\vv^{(1:2)}}}}^\top\;.
	\end{align}
	One can verify from \eqref{gset} that $\matr\Pi^\perp_{\vv^{(1:2)}}\in \RR^{N\times (N-2)}$ is semi-unitary, satisfying $\matr\Pi^\perp_{\vv^{(1:2)}}(\matr\Pi^\perp_{\vv^{(1:2)}})^\top=\Pm^\perp_{\vv^{(1:2)}}$.
	From \eqref{usethis} we then obtain
	\begin{equation}
\Om^\top\psiv^{(1)}=\langle\vv^{(1)},\psiv^{(1)}\rangle\widetilde\vv^{(1)} +\langle\vv^{(2)},\psiv^{(1)}\rangle\widetilde \vv^{(2)}\;,
	\end{equation}
completing the $\Om$-free representation for the first iteration step $t=1$.

	We now address the second iteration step, using notation compatible with the general case $t>1$. We generate independent Gaussian vectors
	\[\widetilde\gv^{(t)}\sim\gv^{(t)}\sim\mathcal {N}(\matr 0,\Id_N),\]
	and construct 
	\begin{align}
		\widetilde\vv^{(2t-1)}&=\mathcal{GS}(\widetilde\gammav^{(t)}\vert \widetilde\vv^{(1:2(t-1))}),\\
		\vv^{(2t-1)}&=\mathcal{GS}(\gv^{(t)}\vert \vv^{(1:2(t-1))})\;.
	\end{align}
	Observe that
	\begin{equation}
		(\matr\Pi_{\vv^{(1:2(t-1))}}^\perp)^\top\gv^{(t)} \sim\mathcal {N}(\matr 0,\Id_{N-2(t-1)})
	\end{equation}
	and is independent of $\vv^{(1:2(t-1))}$. Moreover,
	\begin{align}
		\matr\Pi_{\widetilde\vv^{(1:2(t-1))}}^\perp\mathcal{GS}((\matr\Pi_{\widetilde\vv^{(1:2(t-1))}}^\perp)^\top\widetilde\vv^{(2t-1)})&\equiv \widetilde \vv^{(2t-1)},\\
		\matr\Pi_{\vv^{(1:2(t-1))}}^\perp\mathcal{GS}((\matr\Pi_{\vv^{(1:2(t-1))}}^\perp)^\top\gv^{(t)})
		&\equiv \vv^{(2t-1)}\;.
	\end{align}
	Hence, analogously to \eqref{usethis}, we obtain the representation
	\begin{align}
		\Om&=\frac{1}{N} \sum_{1\leq s\leq 2t-1}\vv^{(s)}(\widetilde\vv^{(s)})^\top+{\matr\Pi^\perp_{\vv^{(1:2t-1)}}}
		\Om_{N-2t+1}({\matr\Pi^\perp_{\widetilde\vv^{(1:2t-1)}}})^\top\;.\label{usethisend1}
	\end{align}
	Since by construction
	\begin{align}
	\widetilde\gammav^{(t)}&=\sum_{1\leq s\leq 2t-1}\langle \widetilde\vv^{(s)},\widetilde\gammav^{(t)} \rangle\widetilde\vv^{(s)},
	\end{align}
	we get from \eqref{usethisend1} that
	\begin{align}
		\Om\widetilde\gammav^{(t)}&=\sum_{1\leq s\leq 2t-1}\langle \widetilde\vv^{(s)},\widetilde\gammav^{(t)} \rangle\vv^{(s)}\;.
	\end{align}

	To obtain the $\Om$-free representation of $\Om^\top\psiv^{(t)}$, we construct
	\begin{align}
		\vv^{(2t)}&=\mathcal{GS}(\psiv^{(t)}\vert \vv^{(1:2t-1)}),\\
		\widetilde\vv^{(2t)}&=\mathcal{GS}(\widetilde\gv^{(t)}\vert \widetilde\vv^{(1:2t-1)})\;.
	\end{align}
	Analogously to \eqref{usethisend1}, we obtain
	\begin{align}
		\Om&=\frac{1}{N} \sum_{1\leq s\leq 2t}\vv^{(s)}(\widetilde\vv^{(s)})^\top+{\matr\Pi^\perp_{\vv^{(1:2t)}}}
		\Om_{N-2t}({\matr\Pi^\perp_{\widetilde\vv^{(1:2t)}}})^\top\;.\label{usethisend2}
	\end{align}
	Since by construction
	\begin{align}
		\psiv^{(t)}&=\sum_{1\leq s\leq 2t}\langle \vv^{(s)},\psiv^{(t)} \rangle\vv^{(s)}\;,
	\end{align}
	it follows from \eqref{usethisend2} that
	\begin{align}
		\Om^\top\psiv^{(t)}&=\sum_{1\leq s\leq 2t}\langle \vv^{(s)},\psiv^{(t)} \rangle \widetilde \vv^{(s)}\;.
	\end{align}

For iteration steps $t=3,4,\ldots$, the same arguments as for $t=2$ apply, completing the proof of Lemma~\ref{lemmaseq}.

\section{The Case of Dirac Measures for $\mathrm{F}_{\mathbf A}$ and/or $\mathrm{F}_{\mathbf B}$}\label{DiracFA}
\subsection{The Case Where ${\rm F}_{\Am}$ is a Dirac Measure}
We consider the case where ${\rm F}_{\Am}$ is a Dirac measure at a nonzero point, whether or not
${\rm F}_{\Bm}$ is a Dirac measure. First, by an abuse of notation, we set
\begin{equation}
\theta_N\equiv \max\left\{1,\frac{1}{\chi_N},\frac{1}{\vert z^N_{\Am,\Bm}(\chi_N)\vert}\right\}\;.
\end{equation}
For the first iteration, regardless of whether ${\rm F}_{\Am}$ or ${\rm F}_{\Bm}$ is a Dirac measure, we have
\begin{align}
\gammav^{(1)}=\sqrt{\mathcal C_{N}^{(1,1)}+\Opn{N^{-\frac 1 2}}}\,\widetilde\gv^{(1)}+\Opn{1}, \label{47}
\end{align}
where $\widetilde\gv^{(1)}\sim \mathcal N(\matr 0,\Id)$ is independent of $\mathcal C_{N}^{(1,1)}$. Here the proof of \eqref{47} is the same as that of $\gammav^{(1)}=\gammav_{\rm e}^{(1)}+\Opn{1}$ but does not invoke the final-step argument that $\widehat{{\mathcal{B}}}_N^{(1,1)}={{\mathcal{B}}}_N^{(1,1)}+\Opn{N^{-\frac 1 2}}$.

Now consider the case where ${\rm F}_{\Am}$ is a Dirac measure at a nonzero point, allowing ${\rm F}_{\Bm}$ to be either a Dirac measure or not. Since ${\rm R}'_{\Am}(s) = 0$, we have $\tau = \rho = 0$, and $\gammav^{(0)}=\sqrt{\tau}\,\widetilde\gv^{(0)}=\matr 0$. We thus obtain
\begin{align}
\frac{\vert\vv^\top\gammav^{(0)}\vert}{\sqrt N}&=0
\label{166}\\
\frac{\norm{\gammav^{(1)}-\gammav^{(0)}}}{\sqrt N}&=\sqrt{\mathcal C_{N}^{(1,1)}+\Opn{N^{-\frac 1 2}}}
= \underbrace{\sqrt{{\mathcal C_{N}^{(1,1)}}}}_{\overset{\rm a.s.}{\rightarrow}\sqrt\tau= 0}+\Opn{N^{-\frac 1 4}}.\label{167}
\end{align}
Repeating the steps in \eqref{pp1}--\eqref{pp2} with $\Opn{N^{-1/4}}$ in place of $\Opn{N^{-1/2}}$ and $p>4q+6$ completes the proof of Theorem~\ref{th1}.

For Theorem~\ref{th2} in this case, $\tau=0$, so $\sqrt{\tau/N}\,\zv=\matr 0$ for any standard Gaussian
$\zv$ independent of $(\Am,\Bm,\uv)$, and the claim reduces to
$\norm{\gammav}/\sqrt N\overset{\rm a.s.}{\to}0$. Since $\gammav^{(0)}=\sqrt\tau\,\widetilde\gv^{(0)}=\matr 0$,
Lemma~\ref{P1} and $\mathds 1_{\mathcal E_N}\gammav=\gammav$ give
\begin{align}
\frac{\norm{\gammav}}{\sqrt N}\leq
\sigma_{\max}\!\left(\mathds{1}_{\mathcal E_N}\big(\Id_N-\Om^\top\Em_A\Om\Em_B\big)^{-1}\right)
\frac{\norm{\gammav^{(1)}-\gammav^{(0)}}}{\sqrt N}\overset{\rm a.s.}{\longrightarrow}0\;,
\end{align}
where the convergence follows from \eqref{167} and \eqref{bounEAB}.

\subsection{The Case Where Only ${\rm F}_{\Bm}$ is a Dirac Measure} 
Now we consider the case that only ${\rm F}_{\Bm}$ is a Dirac measure (at a non-zero point), so that $\tau>0$. 
Note that in this case we have $\mathcal C_N^{(1,1)}\geq \frac{\phi_N(\Em_A^2)}{\chi_N^2}$. 
We therefore redefine $\theta_N$ as
\begin{equation}
\theta_N\equiv \max\left\{1,\frac{1}{\chi_N},\frac{1}{\vert z^N_{\Am,\Bm}(\chi_N)\vert},\frac{1}{\phi_N(\Em_A^2)}\right\}\;,
\end{equation}
so that $1/\mathcal C_N^{(1,1)}=\Opn{1}$. We then have from \eqref{47} 
\begin{align}
\label{169}
\gammav^{(1)}=\sqrt{\mathcal C_N^{(1,1)}+\Opn{N^{-\frac 1 2}}}\widetilde\gv^{(1)}+\Opn{1}=\sqrt{\mathcal C_N^{(1,1)}}\widetilde\gv^{(1)}+\Opn{1}\;.
\end{align}
We further verify below  that
\begin{align}
\gammav^{(2)}=\frac{{\mathcal C}_N^{(1,2)}}{\sqrt{{\mathcal C}_N^{(1,1)}}}\,\widetilde\gv^{(1)}
+\sqrt{\mathcal C_N^{(2,2)}-\frac{(\mathcal C_N^{(2,1)})^2}{\mathcal C_N^{(1,1)}}+\Opn{N^{-\frac 1 2}}}\,\widetilde\gv^{(2)}+\Opn{1},\label{exception}
\end{align}
where $\widetilde\gv^{(2)} \sim\mathcal N(\matr 0,\Id)$ is independent of $\widetilde\gv^{(1)}$ and $\mathcal C_N^{(1:2)}$. Since $\rho=0$ in this case, we have for $t=1,2$,
\[
\mathcal C_N^{(1,2)}-\mathcal C_N^{(t,t)}\overset{\rm a.s.}{\rightarrow}0 \quad \text{as} \quad N\to\infty.
\]
Applying Lemma~\ref{Gaussiian-Con} then yields
\begin{align}
\frac{\vert\vv^\top\gammav^{(1)}\vert}{\sqrt N}&=\sqrt{\mathcal C_{N}^{(1,1)}}\Op{N^{-\frac 1 2}}+ \Opn{N^{-\frac 1 2}},\\
\frac{\norm{\gammav^{(2)}-\gammav^{(1)}}}{\sqrt N}&=\sqrt{\mathcal C_{N}^{(1,1)}+\mathcal C_N^{(2,2)}-2\mathcal C_{N}^{(2,1)}+\Opn{N^{-\frac 1 2}}}\\
&=\underbrace{\sqrt{\mathcal C_{N}^{(1,1)}+\mathcal C_N^{(2,2)}-2\mathcal C_{N}^{(2,1)}}}_{\overset{\rm a.s.}{\rightarrow}\, 0}+\Opn{N^{-\frac 1 4}}\;. \label{173}
\end{align}
Repeating the steps in \eqref{pp1}--\eqref{pp2} with $\Opn{N^{-1/4}}$ in place of $\Opn{N^{-1/2}}$ and $p>4q+6$ completes the proof of Theorem~\ref{th1}.

For Theorem~\ref{th2} in this case, we take $\zv\doteq\widetilde\gv^{(1)}$, which is a standard Gaussian
vector independent of $(\Am,\Bm,\uv)$ by construction. Since $\rho=0$, no limiting dynamics
$\gammav_\infty^{(t)}$ is needed: it suffices to compare $\gammav$ with the first iterate. Indeed, from
Lemma~\ref{P1}, $\mathds 1_{\mathcal E_N}\gammav=\gammav$ and $\eqref{implication}$,
\begin{align}
\frac{\norm{\gammav-\gammav^{(1)}}}{\sqrt N}\leq
\sigma_{\max}\!\left(\mathds{1}_{\mathcal E_N}\big(\Id_N-\Om^\top\Em_A\Om\Em_B\big)^{-1}\right)
\frac{\norm{\gammav^{(2)}-\gammav^{(1)}}}{\sqrt N}
+(1-\mathds 1_{\mathcal E_N})\frac{\norm{\gammav^{(1)}}}{\sqrt N}
\overset{\rm a.s.}{\longrightarrow}0\;,
\end{align}
where the first term vanishes by \eqref{173} together with \eqref{bounEAB}. Moreover,
\eqref{169} gives
\begin{align}
\frac{\norm{\gammav^{(1)}-\sqrt{\tau}\,\zv}}{\sqrt N}
\leq\left\vert\sqrt{\mathcal C_N^{(1,1)}}-\sqrt\tau\right\vert\frac{\norm{\widetilde\gv^{(1)}}}{\sqrt N}
+\frac{\Opn{1}}{\sqrt N}\overset{\rm a.s.}{\longrightarrow}0\;,
\end{align}
since $\mathcal C_N^{(1,1)}\overset{\rm a.s.}{\to}\tau$ when $\rho=0$,
$\norm{\widetilde\gv^{(1)}}/\sqrt N\overset{\rm a.s.}{\to}1$, and the $\Opn{1}/\sqrt N$ term converges to
zero almost surely by Markov's inequality and the Borel--Cantelli lemma, as in
\eqref{pp1}--\eqref{pp2} with $p>2$. Combining the two displays proves Theorem~\ref{th2}.
\subsubsection*{Proof of Equation \eqref{exception} }
The proof proceeds through the following chain of implications $(a)\implies (b)\implies \cdots$:
\begin{align*}
\gammav^{(1)}&\overset{(a)}{=}\gammav_e^{(1)}+\Opn{1},\\
\widetilde\gammav^{(2)}&\overset{(b)}{=}\widetilde\gammav_e^{(2)}+\Opn{1},\\
\qv^{(2)}&\overset{(c)}{=}\frac{\widetilde{\mathcal C}_N^{(1,2)}}{\sqrt{\widetilde{\mathcal C}_N^{(1,1)}}}\,\gv^{(1)}+\sqrt{\widetilde{\mathcal C}_N^{(2,2)}-\frac{(\widetilde{\mathcal C}_N^{(1,2)})^2}{\widetilde{\mathcal C}_N^{(1,1)}}+\Opn{N^{-\frac 1 2}}}\,\gv^{(2)}+\Opn{1},\\
\psiv^{(2)}&\overset{(d)}{=}\frac{\widetilde{\mathcal C}_N^{(1,2)}}{\sqrt{\widetilde{\mathcal C}_N^{(1,1)}}}\Em_{A}\gv^{(1)}+\sqrt{\widetilde{\mathcal C}_N^{(2,2)}-\frac{(\widetilde{\mathcal C}_N^{(1,2)})^2}{\widetilde{\mathcal C}_N^{(1,1)}}+\Opn{N^{-\frac 1 2}}}\Em_{A}\gv^{(2)}+\Opn{1},\\
\langle\gv^{(s)},\psiv^{(2)} \rangle&\overset{(e)}{=}\Opn{N^{-\frac 1 2}}, \quad s=1,2, \\
\deltav_\gamma^{(2)}&\overset{(f)}{=}\Opn{1}, \\
\psiv_\perp^{(1:2)}&\overset{(f)}{=}\psiv^{(1:2)}+\Opn{1},\\
\widehat{\mathcal C}_N^{(1:2)}&\overset{(g)}{=}{\mathcal C}_N^{(1:2)}+\Opn{N^{-\frac 1 2}}\;.
\end{align*}
We now recall that $1/\mathcal C_N^{(1,1)}=\Opn{1}$ and $\mathcal C_N^{(1:2)}=\Opn{1}$. Hence $(g)$ together with the arithmetic properties of $\Opn{\kappa}$ in \eqref{arit} give
\begin{align}
\widehat{\mathcal B}_N^{(1,2)}=\frac{{\mathcal C}_N^{(1,2)}}{\sqrt{{\mathcal C}_N^{(1,1)}}}+\Opn{N^{-\frac 1 2}},
\qquad
\widehat{\mathcal B}_N^{(2,2)}=\sqrt{\mathcal C_N^{(2,2)}-\frac{(\mathcal C_N^{(2,1)})^2}{\mathcal C_N^{(1,1)}}+\Opn{N^{-\frac 1 2}}}\;. 
\end{align}
Substituting these two coefficients and $\deltav_\gamma^{(2)}=\Opn{1}$ from $(f)$
into \eqref{phiv} with $t=2$, and recalling $\norm{\widetilde\gv^{(s)}}=\Opn{\sqrt N}$,
yields \eqref{exception}.
\bibliographystyle{IEEEtran}
\bibliography{report}
\end{document}